\documentclass{mfatshort}
\usepackage{amsmath}
\usepackage{enumerate}
\usepackage{amsmath,amsthm,amscd,amssymb}

\usepackage{latexsym}
\usepackage{upref}
\usepackage{verbatim}

\usepackage[mathscr]{eucal}
\usepackage{dsfont}

\usepackage{graphicx}
\usepackage[colorlinks,hyperindex,pagebackref,hypertex]{hyperref}

\newtheorem{theorem}{Theorem}
\newtheorem{proposition}[theorem]{Proposition}
\newtheorem{corollary}[theorem]{Corollary}
\newtheorem{lemma}[theorem]{Lemma}
\newtheorem{definition}[theorem]{Definition}
\newtheorem{remark}[theorem]{Remark}
\newtheorem{hypothesis}[theorem]{Hypothesis}

\chardef\bslash=`\\ 

\newcommand{\wh}{\widehat}

\newcommand{\whA}{T}
\newcommand{\whB}{T_{\cB}}
\newcommand{\whBo}{T_{\cB_0}}

\newcommand{\dA}{{\dot A}}

\newcommand{\bbR}{{\mathbb{R}}}

\newcommand{\bbD}{{\mathbb{D}}}

\newcommand{\bbC}{{\mathbb{C}}}

\newcommand{\ran}{\text{\rm{Ran}}}

\newcommand{\ol}{\overline}
\newcommand{\ti}{\tilde  }

\newcommand{\dom}{\text{\rm{Dom}}}
\newcommand{\Dom}{\text{\rm{Dom}}}

\newcommand{\calH}{{\mathcal H}}

\newcommand{\calK}{{\mathcal K}}

\newcommand{\calM}{{\mathcal M}}

\newcommand{\calR}{{\mathcal R}}

\newcommand{\linspan}{\mathrm{lin\ span}}

\newcommand{\mM}{\mathfrak M}

\newcommand{\DdA}{\dom(\dA)}

\renewcommand{\Im}{\text{\rm Im}}

\def\sJ{{\mathfrak J}}      \def\sL{{\mathfrak L}}
\def\sM{{\mathfrak M}}   \def\sN{{\mathfrak N}}

\def\bA{{\mathbb A}}   \def\dB{{\mathbb B}}   \def\dC{{\mathbb C}}

      \def\dR{{\mathbb R}}

   \def\cB{{\mathcal B}}   
      
   \def\cH{{\mathcal H}}

\def\RE{{\rm Re\,}}
\def\Ker{{\rm Ker\,}}

\def\wh{\hat}

\def\uphar{{\upharpoonright\,}}

\DeclareMathOperator{\IM}{Im}

\newcommand{\eval}[2][\right]{\relax
  \ifx#1\right\relax \left.\fi#2#1\rvert}

\begin{document}

\title{Conservative  L-systems and the Liv\v{s}ic function}

\author{S. Belyi}
\address{Department of Mathematics\\ Troy State University\\
Troy, AL 36082, USA\\
}
\curraddr{}
\email{sbelyi@troy.edu}

\author{K. A. Makarov}
\address{Department of Mathematics, University of Missouri, Columbia, MO 65211, USA}
\email{makarovk@missouri.edu}

\author{E. Tsekanovski\u i}
\address{Department of Mathematics\\ Niagara University, NY 14109\\ USA}
\email{tsekanov@niagara.edu}


\dedicatory{Dedicated to Yury Berezansky, a remarkable Mathematician and Human Being, on the occasion of his $90^{th}$ birthday.}


\begin{abstract}
We study the connection between  {the classes of (i)} Liv\v{s}ic  functions $s(z),$ { i.e.,} the characteristic functions  of  densely defined symmetric operators $\dot A$ with deficiency indices $(1, 1)$;
{(ii)} the characteristic functions $S(z)$  of a maximal dissipative extension $T$ of $\dot A,$ { i.e.,  the M\"obius transform of $s(z)$ determined by the von Neumann parameter $\kappa$ of the extension relative to an appropriate basis in the deficiency subspaces};
and  {(iii)} the transfer functions $W_\Theta(z)$ of a conservative L-system $\Theta$ with the main operator $T$. It is shown that under a natural hypothesis  {the functions} $S(z)$ and $W_\Theta(z)$ are reciprocal to each other. In particular,  $W_\Theta(z)=\frac{1}{S(z)}=-\frac{1}{s(z)}$ whenever $\kappa=0$. It is established that the impedance function of a conservative L-system with the main operator $T$ belongs to the Donoghue class if and only if the von Neumann parameter  vanishes ($\kappa=0$). Moreover, we introduce the generalized Donoghue class and obtain the criteria for an impedance function to belong to this class. We also obtain the representation of a function from this class via the Weyl-Titchmarsh function. All results are illustrated by a number of examples.
 \end{abstract}

\maketitle

\section{Introduction}\label{s1}

Suppose that $T$ is a densely defined closed operator in a Hilbert space $\cH$ such that its resolvent set $\rho(T)$ is not empty  {and} assume, in addition, that
$\Dom(T)\cap \Dom(T^*)$ is dense.  {We also suppose  that} the restriction $\dA=T|_{\Dom(T)\cap \Dom(T^*)}$ is a closed symmetric operator with finite equal deficiency indices
 { and that}  $\calH_+\subset\calH\subset\calH_-$  {is}  the rigged Hilbert space associated with $\dot A$ (see Appendix \ref{A1} for a detailed discussion of a concept of  rigged Hilbert spaces).

One of the main objectives  of the current paper is the study of the \textit{L-system}
\begin{equation}
\label{col0}
 \Theta =
\left(%
\begin{array}{ccc}
  \bA    & K & J \\
   \calH_+\subset\calH\subset\calH_- &  & E \\
\end{array}%
\right),
\end{equation}
where the \textit{state-space operator} $\bA$ is a bounded linear operator from
$\calH_+$ into $\calH_-$ such that  $\dA \subset T\subset \bA$, $\dA \subset T^* \subset \bA^*$, $E$ is a finite-dimensional Hilbert space,
$K$ is a bounded linear operator from the space $E$ into $\calH_-$,   {and} $J=J^*=J^{-1}$ is a self-adjoint isometry  on $E$ such that the imaginary part of $\bA$ has a representation $\IM\bA=KJK^*$. Due to the facts that $\calH_\pm$ is dual to $\calH_\mp$ and  {that} $\bA^*$ is a bounded linear operator from $\calH_+$ into $\calH_-$, $\IM\bA=(\bA-\bA^*)/2i$ is a well defined bounded operator from $\calH_+$  into $\calH_-$.
Note that the main operator $T$ associated with the system $\Theta$ is uniquely determined by the state-space operator $\bA$ as its restriction on the domain $\dom(T)=\{f\in\calH_+\mid \bA f\in\calH\}$.

Recall that the operator-valued function  given by
\begin{equation*}\label{W1}
 W_\Theta(z)=I-2iK^*(\bA-zI)^{-1}KJ,\quad z\in \rho(T),
\end{equation*}
 is called the \textit{transfer function}  of the L-system $\Theta$ and
\begin{equation*}\label{real2}
 V_\Theta(z)=i[W_\Theta(z)+I]^{-1}[W_\Theta(z)-I] =K^*(\RE\bA-zI)^{-1}K,\quad z\in\rho(T)\cap\dC_{\pm},
\end{equation*}
is called the \textit{impedance function } of $\Theta$.

 {We remark that under the hypothesis   $\IM\bA=KJK^*$, the linear sets
$\ran(\mathbb  A-zI)$ and $\ran(\RE\mathbb  A-zI)$ contain $\ran (K)$ for   $z\in \rho(T)$
 and $z\in\rho(T)\cap\dC_{\pm}$, respectively, and therefore, both the transfer and impedance functions are well defined (see Section 2 for more details).}

  Note that if $\varphi_+=W_{\Theta}(z)\varphi_-$, where $\varphi_\pm\in E$, with $\varphi_-$ the input and $\varphi_+$ the output, then L-system \eqref{col0} can be associated with the
 system of equations
 \begin{equation}\label{Lyv}
\left\{   \begin{array}{l}
          (\bA-zI)x=KJ\varphi_-  \\
          \varphi_+=\varphi_- -2iK^* x
          \end{array} {.}
\right .
\end{equation}
  ({To recover $W_\Theta(z)\varphi_-$ from \eqref{Lyv} for a given $\varphi_-$, one needs to find $x$ and then determine $\varphi_+$.)}

 We remark that the concept of L-systems \eqref{col0}-\eqref{Lyv}  generalizes the one of the Liv\u sic systems in the case of a bounded main operator. It is also worth mentioning that
 those systems are   conservative in the sense  that a certain metric conservation law holds (for more details see \cite[Preface]{ABT}). An overview of the history of the subject and a detailed description  of the $L$-systems can be found in \cite{ABT}.

Another important object of interest  {in this context} is the \textit{Liv\v{s}ic function}. Recall that in \cite{L} M. Liv\v{s}ic introduced a fundamental concept of a characteristic function
of a densely defined symmetric operator $\dA$ with deficiency indices $(1, 1)$  as well as  of  its maximal non-self-adjoint  extension $T$.
 {Introducing an auxiliary self-adjoint (reference)  extension $A$  of $\dot A$,  in  \cite{MT-S}  two of the authors  (K.A.M. and E.T.) suggested to define   {a } characteristic function  of a symmetric operator {as well} of its dissipative extension  {as the one} associated with  the pairs $(\dot A, A)$  and  $(T, A)$, rather than with the single operators $\dot A$ and $T$, respectively. }
Following \cite{MT-S} and \cite{MT10} we call the characteristic function associated with the pair $(\dot A, A)$ the \textit{Liv\v{s}ic function}.
 For a detailed treatment  of the aforementioned  concepts of the   Liv\v{s}ic and the characteristic functions   we refer to \cite{MT-S} (see also \cite{AWT}, \cite{DM}, \cite{Koch}, \cite{St68}, \cite{TSh1}).

The main goal of the present paper is the following.

First, we establish a connection between  {the classes of}: (i) the Liv\v{s}ic functions $s(z)$, the characteristic functions  of a densely defined symmetric operators $\dot A$ with deficiency indices $(1, 1)$; (ii) the characteristic functions $S(z)$ of a maximal dissipative extension $T$ of $\dot A$,
 {the M\"obius transform of $s(z)$ determined by the von Neumann parameter $\kappa$}; and (iii) the transfer functions $W_\Theta(z)$ of an  L-system $\Theta$ with the main operator $T$. It is shown (see Theorem \ref{t-25}) that under some natural assumptions  {the functions} $S(z)$ and $W_\Theta(z)$ are reciprocal to each other. In particular, when $\kappa=0$, we have $W_\Theta(z)=\frac{1}{S(z)}=-\frac{1}{s(z)}$.

 Second, in Theorem \ref{t-21}, we show  that the impedance function of a conservative L-system with the main operator $T$ coincides with a function from  the Donoghue class $\sM$ if and only if the von Neumann parameter  vanishes that is $\kappa=0$.  For $0\le\kappa<1$ we introduce the generalized Donoghue class $\sM_\kappa$ and establish a {criterion}
  (see Theorem \ref{t-22}) for an impedance function to belong to $\sM_\kappa$.  In particular, when $\kappa=0$ the class $\sM_\kappa$ coincides with the Donoghue class $\sM=\sM_0$. Also, in Theorem \ref{t-14}, we  obtain the representation of a function from the class $\sM_\kappa$ via the Weyl-Titchmarsh function.

 We conclude  our paper  {by providing} several examples that illustrate  the main results and concepts.

\section{Preliminaries}\label{s2}

For a pair of Hilbert spaces $\calH_1$ and $\calH_2$ we denote by
$[\calH_1,\calH_2]$ the set of all bounded linear operators from
$\calH_1$ to $\calH_2$. Let $\dA$ be a closed, densely defined,
symmetric operator in a Hilbert space $\calH$ with inner product
$(f,g),f,g\in\calH$. Any operator $T$ in $\cH$ such that
\[
\dA\subset T\subset\dA^*
\]
is called a \textit{quasi-self-adjoint extension} of $\dA$.

 Consider the rigged Hilbert space (see \cite{Ber63}, \cite{Ber}, \cite{BT3})
$\calH_+\subset\calH\subset\calH_- ,$ where $\calH_+ =\dom(\dA^*)$ and
\begin{equation}\label{108}
(f,g)_+ =(f,g)+(\dA^* f, \dA^*g),\;\;f,g \in \dom(\dA^*).
\end{equation}
Let $\calR$ be the \textit{\textrm{Riesz-Berezansky   operator}} $\calR$ (see \cite{Ber63}, \cite{Ber}, \cite{BT3}) which maps $\mathcal H_-$ onto $\mathcal H_+$  {so}
 that   $(f,g)=(f,\calR g)_+$ ($\forall f\in\calH_+$, $\forall g\in\calH_-$) and
 $\|\calR g\|_+=\| g\|_-$.
 Note that
identifying the space conjugate to $\calH_\pm$ with $\calH_\mp$, we
get that if $\bA\in[\calH_+,\calH_-]$, then
$\bA^*\in[\calH_+,\calH_-].$

  {Next we proceed with several definitions.}

An operator $\bA\in[\calH_+,\calH_-]$ is called a \textit{self-adjoint
bi-extension} of a symmetric operator $\dA$ if $\bA=\bA^*$ and $\bA
\supset \dA$.

Let $\bA$ be a self-adjoint
bi-extension of $\dA$ and let the operator $\hat A$ in $\cH$ be defined as follows:
\[
\dom(\hat A)=\{f\in\cH_+:\bA f\in\cH\}, \quad \hat A=\bA\uphar\dom(\hat A).
\]
The operator $\hat A$ is called a \textit{quasi-kernel} of a self-adjoint bi-extension $\bA$ (see \cite{TSh1}, \cite[Section 2.1]{ABT}).

 A   self-adjoint bi-extension $\bA$ of a symmetric operator $\dA$ is called twice-self-adjoint or \textit{t-self-adjoint} (see \cite[Definition 3.3.5]{ABT}) if its quasi-kernel $\hat A$ is a
self-adjoint operator in $\calH$.  {In this case},
according to the von Neumann Theorem (see \cite[Theorem 1.3.1]{ABT}) the domain of $\wh A$,    {which is}
a self-adjoint extension of $\dA$,  can be represented as
\begin{equation}\label{DOMHAT}
\dom(\hat A)=\dom(\dA)\oplus(I+U)\sN_{i},
\end{equation}
where $U$ is both a $(\cdot)$-isometric as well as $(+)$)-isometric operator from $\sN_i$ into $\sN_{-i}$.
 Here $$\sN_{\pm i}=\Ker (\dA^*\mp i I)$$ are the deficiency subspaces of $\dA$.

An operator $\bA\in[\calH_+,\calH_-]$  is called a \textit{quasi-self-adjoint bi-extension} of an operator $T$ if
   $\dA\subset T\subset \bA$ and $\dA\subset T^*\subset \bA^*.$

In what follows
 we will be mostly interested in the following type of quasi-self-adjoint bi-extensions.
\begin{definition}[\cite{ABT}]\label{star_ext}
Let $T$ be a quasi-self-adjoint extension of $\dA$ with nonempty
resolvent set $\rho(T)$. A quasi-self-adjoint bi-extension $\bA$ of an operator $T$ is called a \textit{($*$)-extension }of $T$ if $\RE\bA$ is a
t-self-adjoint bi-extension of $\dA$.
\end{definition}
We assume that $\dA$ has equal finite deficiency indices and will say that a quasi-self-adjoint extension $T$ of $\dA$ belongs to the
\textit{class $\Lambda(\dA)$} if $\rho(T)\ne\emptyset$, $\dom(\dA)=\dom(T)\cap\dom(T^*)$, and hence  $T$ admits $(*)$-extensions. The description of
all $(*)$-extensions via Riesz-Berezansky   operator $\calR$ can be found in \cite[Section 4.3]{ABT}.

\begin{definition} 
A system of equations
\[
\left\{   \begin{array}{l}
          (\bA-zI)x=KJ\varphi_-  \\
          \varphi_+=\varphi_- -2iK^* x
          \end{array}
\right.,
\]
 or an
array
\begin{equation}\label{e6-3-2}
\Theta= \begin{pmatrix} \bA&K&\ J\cr \calH_+ \subset \calH \subset
\calH_-& &E\cr \end{pmatrix}
\end{equation}
 is called an \textbf{{L-system}}   if:
\begin{itemize}
\item[(1)] {$\mathbb  A$ is a   ($\ast $)-extension of an
operator $T$ of the class $\Lambda(\dA)$};
\item[(2)] {$J=J^\ast =J^{-1}\in [E,E],\quad \dim E < \infty $};
\item[(3)] $\IM\bA= KJK^*$, where $K\in [E,\calH_-]$, $K^*\in [\calH_+,E]$, and
$\ran(K)=\ran (\IM\bA).$
\end{itemize}
\end{definition}

{In what follows we assume the following terminology}. 
In the definition above   $\varphi_- \in E$ stands for an input vector, $\varphi_+ \in E$ is an output vector, and $x$ is a state space vector in
$\calH$.
The operator $\bA$  is called the \textit{state-space operator} of the system $\Theta$, $T$ is the \textit{main operator},  $J$ is the \textit{direction operator}, and $K$ is the  \textit{channel operator}. A system $\Theta$ \eqref{e6-3-2} is called \textit{minimal} if the operator $\dA$ is a prime operator in $\calH$, i.e., there exists no non-trivial
 subspace invariant for $\dot A$ such that the restriction of $\dot A$ to this subspace is self-adjoint.

We  associate with an L-system $\Theta$ the operator-valued function
\begin{equation}\label{e6-3-3}
W_\Theta (z)=I-2iK^\ast (\mathbb  A-zI)^{-1}KJ,\quad z\in \rho (T),
\end{equation}
which is called the \textbf{transfer  function} of the L-system $\Theta$. We also consider the operator-valued function
\begin{equation}\label{e6-3-5}
V_\Theta (z) = K^\ast (\RE\bA - zI)^{-1} K, \quad z\in\rho(\hat A).
\end{equation}

It was shown in \cite{BT3}, \cite[Section 6.3]{ABT} that both \eqref{e6-3-3} and \eqref{e6-3-5} are well defined. In particular,  $\ran(\mathbb  A-zI)$   does not depend on $z\in\rho(T)$ while $\ran(\RE\mathbb  A-zI)$ does not depend on $z\in\rho(\hat A)$. Also, $\ran(\mathbb  A-zI)\supset\ran(K)$ and $\ran(\RE\mathbb  A-zI)\supset\ran(K)$ (see  \cite[Theorem 4.3.2]{ABT}). The transfer operator-function $W_\Theta (z)$ of the system
$ \Theta $ and an operator-function $V_\Theta (z)$ of the form (\ref{e6-3-5}) are connected by the following relations valid for $\IM z\ne0$, $z\in\rho(T)$,
\begin{equation}\label{e6-3-6}
\begin{aligned}
V_\Theta (z) &= i [W_\Theta (z) + I]^{-1} [W_\Theta (z) - I] J,\\
W_\Theta(z)&=(I+iV_\Theta(z)J)^{-1}(I-iV_\Theta(z)J).
\end{aligned}
\end{equation}
The function $V_\Theta(z)$ defined by \eqref{e6-3-5} is called the
\textbf{impedance function} of the L-system $ \Theta $.
 The class of all Herglotz-Nevanlinna functions in a finite-dimensional Hilbert
space $E$, that can be realized as impedance functions of an L-system, was described in \cite{BT3} (see also \cite[Definition 6.4.1]{ABT}).

Two minimal L-systems
$$
\Theta_j= \begin{pmatrix} \bA_j&K_j&J\cr \calH_{+j} \subset\calH_j \subset \calH_{-j}& &E\cr
\end{pmatrix} \quad j=1,2.
$$
are called \textbf{bi-unitarily equivalent} \cite[Section 6.6]{ABT} if there exists a triplet of operators $(U_+, U, U_-)$ that isometrically maps the triplet $\calH_{+1}\subset\calH_1\subset\calH_{-1}$ onto the triplet $\calH_{+2}\subset\calH_2\subset\calH_{-2}$
such that $U_+=U|_{\calH_{+1}}$ is an isometry from $\calH_{+1}$ onto $\calH_{+2}$, $U_-=(U_+^*)^{-1}$ is an isometry from $\calH_{-1}$ onto $\calH_{-2}$, and
\begin{equation}\label{167}
UT_1=T_2U,  \quad U_-\bA_1=\bA_2 U_+,\quad U_-K_1=K_2.
\end{equation}
It is shown in \cite[Theorem 6.6.10]{ABT} that if the transfer functions $W_{\Theta_1}(z)$ and $W_{\Theta_2}(z)$ of the minimal systems $\Theta_1$ and $\Theta_2$ coincide for
$ z\in(\rho(T_1)\cap\rho(T_2))\cap \dC_{\pm}\ne\emptyset$, then $\Theta_1$ and $\Theta_2$ are bi-unitarily equivalent.


\section{On $(*)$-extension parametrization}\label{s3}

Let $\dA$ be a densely defined, closed, symmetric operator with finite deficiency indices $(n,n)$. Then  (see \cite[Section 2.3]{ABT})
$$
\calH_+=\dom(\dA^*)=\DdA\oplus\sN_i\oplus\sN_{-i},
$$
where $\oplus$ stands for the $(+)$-orthogonal sum. Moreover,
all operators from the class $\Lambda(\dA)$ are of the form (see \cite[Theorem 4.1.9]{ABT}, \cite{TSh1})
\begin{equation}\label{e4-60}
    \begin{aligned}
\dom(T)&=\dom(\dA)\oplus(\calK +I)\sN_{i},\quad T=\dA^*\uphar\dom(T),\\
\dom(T^*)&=\dom(\dA)\oplus(\calK^*+I)\sN_{-i},\quad
T^*=\dA^*\uphar\dom(T^*),
    \end{aligned}
\end{equation}
where $\calK\in[\sN_i,\sN_{-i}]$.

Let $\calM=\sN_i\oplus\sN_{-i}$ and $P^+_\sN$ be a $(+)$-orthogonal projection  onto a  subspace $\sN$.
 In this case (see \cite{TSh1}) all quasi-self-adjoint bi-extensions of $T\in\Lambda(\dA)$ can be parameterized via  an  operator $H\in[\sN_{-i},\sN_i]$ as follows
\begin{equation}\label{e4-61}
\bA=\dA^*+\calR^{-1} (S-\cfrac{i}{2}\sJ)P_\calM^+,\quad
\bA^*=\dA^*+\calR^{-1} (S^*-\cfrac{i}{2}\sJ)P_\calM^+,
\end{equation}
where $\sJ=P^+_{\sN_i}-P^+_{\sN_{-i}}$ and $S:\sN_i\oplus\sN_{-i}\rightarrow\sN_i\oplus\sN_{-i}$, satisfies the condition
\begin{equation}\label{opH}
S=\begin{pmatrix}\frac{i}{2}I-H\calK&H\cr
-(iI-\calK H)\calK&\frac{i}{2}I-\calK H
\end{pmatrix}.
\end{equation}
 Introduce $(2n\times2n)$-- block-operator matrices $S_{\bA}$ and $S_{\bA^*}$ by
\begin{equation}\label{e4-62}
\begin{aligned}
    S_\bA&=S-\frac{i}{2}\sJ=\left(
            \begin{array}{cc}
              -H\calK & H \\
              \calK(H\calK-iI) & iI-\calK H \\
            \end{array}
          \right),\\
S_{\bA^*}&=S^*-\frac{i}{2}\sJ=\left(
            \begin{array}{cc}
              -\calK^*H^*-iI & (\calK^*H^*-iI)\calK^* \\
              H^* & -H^*\calK^* \\
            \end{array}
          \right).
\end{aligned}
\end{equation}
By direct calculations one finds that
\begin{equation}\label{e-12-real}
    \frac{1}{2}(S_{\bA}+S_{\bA^*})=\frac{1}{2}\left(
            \begin{array}{cc}
              -H\calK-\calK^*H^*-iI & H+(\calK^*H^*+iI)\calK^* \\
              \calK(H\calK-iI)+H^* & iI-\calK H-H^*\calK^* \\
            \end{array}
          \right),
\end{equation}
and  {that}
\begin{equation}\label{e10}
    \frac{1}{2i}(S_{\bA}-S_{\bA^*})=\frac{1}{2i}\left(
            \begin{array}{cc}
              -H\calK+\calK^*H^*+iI & H-(\calK^*H^*+iI)\calK^* \\
              \calK(H\calK-iI)-H^* & iI-\calK H+H^*\calK^* \\
            \end{array}
          \right).
\end{equation}
In the case when the deficiency indices of $\dA$ are $(1,1)$, the  block-operator matrices $S_{\bA}$ and $S_{\bA^*}$ in \eqref{e4-62} become {$(2\times2)-$matrices}
 with scalar entries.
As it was announced in \cite{T69}, (see also \cite[Section 3.4]{ABT} and \cite{TSh1}), in this case any quasi-self-adjoint bi-extension $\bA$ of $T$ is of the form
\begin{equation}\label{e3-40}
    \bA=\dA^*+\left[p(\cdot,\varphi)+q(\cdot,\psi) \right]\varphi
    +\left[v(\cdot,\varphi)+w(\cdot,\psi) \right]\psi,
\end{equation}
where $S_{\bA}=\left(\begin{array}{cc}
                                        p & q \\
                                        v & w \\
                                      \end{array}
                                    \right)$
is a $(2\times2)$ -- matrix with scalar entries such that $p=-H\calK$, $q=H$, $v=\calK(H\calK-i)$, and $w=i-\calK H$. Also, $\varphi$ and $\psi$ are $(-)$-normalized elements in $\calR^{-1}(\sN_i)$ and  $\calR^{-1}(\sN_{-i})$, respectively. Both the parameters $H$ and $\calK$ are complex numbers in this case and $|\calK|<1$. Similarly we write
\begin{equation}\label{e-21-star}
    \bA^*=\dA^*+\left[p^\times(\cdot,\varphi)+q^\times(\cdot,\psi) \right]\varphi
    +\left[v^\times(\cdot,\varphi)+w^\times(\cdot,\psi) \right]\psi,
\end{equation}
where $S_{\bA^*}=\left(\begin{array}{cc}
                                        p^\times & q^\times \\
                                        v^\times & w^\times \\
                                      \end{array}
                                    \right)$
is  such that $p^\times=-\bar\calK\bar H-i$, $q^\times=(\bar\calK\bar H-i)\bar\calK$, $v^\times=\bar H$, and $w^\times=-\bar H\bar\calK$.
A direct check confirms that $\dA\subset T\subset\bA$ and we make the corresponding calculations below for the reader's convenience.

Indeed,  recall that $\|\varphi\|_-=\|\psi\|_-=1$. Using formulas \eqref{e3-4} and \eqref{e3-5} from Appendix A we get
$$
1=(\varphi,\varphi)_-=(\calR \varphi,\calR\varphi)_+=\|\calR \varphi\|_+^2=2\|\calR \varphi\|^2=(\sqrt2\calR\varphi,\sqrt2\calR\varphi).
$$
Set $g_+=\sqrt2\calR\varphi\in\sN_i$ and $g_-=\sqrt2\calR\psi\in\sN_{-i}$ and note that $g_+$ and $g_-$ form normalized vectors in $\sN_i$ and $\sN_{-i}$, respectively.
Now let $f\in\dom(T)$, where $\dom(T)$ is defined in \eqref{e4-60}. Then,
\begin{equation}\label{e-22-dec}
f=f_0+(\calK+1)f_1=f_0+Cg_++\calK Cg_-, \quad f_0\in\DdA,\; f_1\in \sN_i,
\end{equation}
for some choice of the constant $C$ that is specific to $f\in\dom(T)$.
Moreover,
$$
\bA f=T f+\left[p(f,\varphi)+q(f,\psi) \right]\varphi +\left[v(f,\varphi)+w(f,\psi) \right]\psi,\quad f\in\dom(T).
$$
Let us show that the last two terms in the sum above vanish.  Consider $(f,\varphi)$ where $f$ is decomposed into the $(+)$-orthogonal sum \eqref{e-22-dec}.
Using $(+)$-orthogonality of $\sN_i$ and $\sN_{-i}$ we have
$$
\begin{aligned}
(f,\varphi)&=(f_0+Cg_++\calK Cg_-,\varphi)=(f_0,\varphi)+(C g_+,\varphi)+(\calK C g_-,\varphi)\\
&=0+(C g_+,\calR\varphi)_++(\calK C g_-,\calR\varphi)_+\\
&=(C g_+,(1/\sqrt2)g_+)_++(\calK C g_-,(1/\sqrt2)g_+)_+\\
&=\frac{C}{\sqrt2}(g_+,g_+)_+=\frac{C}{\sqrt2}\|g_+\|_+^2=\sqrt2 C\|g_+\|^2=\sqrt2 C.
\end{aligned}
$$
Similarly,
$$
\begin{aligned}
(f,\psi)&=(f_0+Cg_++\calK Cg_-,\psi)=(f_0,\psi)+(C g_+,\psi)+(\calK C g_-,\psi)\\
&=0+(C g_+,\calR\psi)_++(\calK C g_-,\calR\psi)_+\\
&=(C g_+,(1/\sqrt2)g_-)_++(\calK C g_-,(1/\sqrt2)g_-)_+\\
&=\frac{\calK C}{\sqrt2}(g_-,g_-)_+=\frac{\calK C}{\sqrt2}\|g_-\|_+^2=\sqrt2 \calK C\|g_-\|^2=\sqrt2 \calK C.
\end{aligned}
$$
Consequently,
$$
p(f,\varphi)+q(f,\psi)=-H\calK(f,\varphi)+H(f,\psi)=H[-\calK\sqrt2 C+\sqrt2 \calK C]=0.
$$
Applying similar argument for the last bracketed term in \eqref{e3-40} we show that
$$
v(f,\varphi)+w(f,\psi)=0
$$
as well. Thus, $\dA\subset T\subset\bA$. Likewise, using \eqref{e-21-star} one shows that $\dA\subset T^*\subset\bA^*$.

The following proposition was announced  by one of the authors (E.T.) in \cite{TSh1} and we present its proof below for convenience of the reader.
\begin{proposition}\label{t3-21}
Let $T\in\Lambda(\dA)$ and $A$ be a self-adjoint extension of $\dA$ such that $U$ defines $\dom(A)$ via \eqref{DOMHAT} and $\calK$ defines $T$ via \eqref{e4-60}.
 Then $\bA$ is a $(*)$-extension of $T$ whose real part $\RE\bA$ has the quasi-kernel $A$ if and only if $U\calK^*-I$ is a homeomorphism and the operator parameter $H$ in \eqref{opH}-\eqref{e4-62} takes the form
\begin{equation}\label{e3-39-new}
H=i(I-\calK^*\calK)^{-1}[(I-\calK^*U)(I-U^*\calK)^{-1}-\calK^*U]U^*.
\end{equation}
\end{proposition}
\begin{proof}
First, we are going to show that $\RE\bA$ has the quasi-kernel $A$ if and only if the system of operator equations
\begin{equation}\label{e-System}
    \begin{aligned}
    X^*(I-\ti\calK^*)+\ti\calK X(\ti\calK-I)&=i(\ti\calK-I)\\
    \ti\calK^* X^*(\ti\calK^*-I)+X(I-\ti\calK)&=i(I-\ti\calK^*)
    \end{aligned}
\end{equation}
has a solution. Here $\ti\calK=U^*\calK$. Suppose  $\RE\bA$ has the quasi-kernel $A$ and $U$ defines $\dom(A)$ via \eqref{DOMHAT}. Then there exists a self-adjoint operator $H\in[\sN_{-i},\sN_i]$ such that $\bA$ and $\bA^*$ are defined via \eqref{e4-61} where $S_{\bA}$ and  $S_{\bA^*}$ are of the form \eqref{e4-62}. Then $\frac{1}{2}(S_{\bA}+S_{\bA^*})$ is given by \eqref{e-12-real}. According to  \cite[Theorem 3.4.10]{ABT} the entries of the operator matrix \eqref{e-12-real} are related by the following
$$
\begin{aligned}
-H\calK-\calK^*H^*-iI &=- (H+(\calK^*H^*+iI)\calK^*)U, \\
              \calK(H\calK-iI)+H^* &=-( iI-\calK H-H^*\calK^*)U.
\end{aligned}
$$
Denoting $\ti\calK=U^*\calK$ and $\ti H=H U$, we obtain
$$
\begin{aligned}
\ti H^*(I-\ti\calK^*)+\ti\calK\ti H(\ti\calK-I) &=i(\ti\calK-I), \\
\ti\calK^*\ti H^*(\ti\calK^*-I)+\ti H(I-\ti\calK)&=i(I-\ti\calK^*),
\end{aligned}
$$
and hence $\ti H$ is the solution to the system \eqref{e-System}. To show the converse we simply reverse the argument.

Now assume that $U\calK^*-I$ is a homeomorphism. We are going to prove that the operator $T$ from the statement of the theorem has a unique $(*)$-extension $\bA$ whose real part $\RE\bA$ has the quasi-kernel $A$ that is a self-adjoint extension of $\dA$ parameterized via $U$. Consider the system \eqref{e-System}. If we multiply the first equation of \eqref{e-System} by $\ti\calK^*$ and add it to the second, we obtain
$$
(I-\ti\calK^*\ti\calK)X(I-\ti\calK)=i(\ti\calK^*(\ti\calK-I)+(I-\ti\calK^*)).
$$
Since $I-\ti\calK^*\ti\calK=I-\calK^*\calK$, $I-\ti\calK^*=I-U^*\calK$, and $T\in\Lambda(\dA)$, then the operators $I-\ti\calK^*\ti\calK$ and $I-\ti\calK$ are boundedly invertible. Therefore,
\begin{equation}\label{e-18-X}
    X=i(I-\ti\calK^*\ti\calK)^{-1}[(I-\ti\calK^*)(I-\ti\calK)^{-1}-\ti\calK^*)].
\end{equation}
By the direct substitution one confirms that the operator $X$ in \eqref{e-18-X} is a solution to the system \eqref{e-System}. Applying the uniqueness result \cite[Theorem 4.4.6]{ABT} and the above reasoning we conclude that our operator $T$ has a unique $(*)$-extension $\bA$ whose real part $\RE\bA$ has the quasi-kernel $A$.
If, on the other hand, $\bA$ is a $(*)$-extension whose real part $\RE\bA$ has the quasi-kernel $A$ that is a self-adjoint extension of $\dA$ parameterized via $U$, then $U\calK^*-I$ is a homeomorphism (see \cite[Remark 4.3.4]{ABT}).

Combining the two parts of the proof,  replacing $\ti\calK$ with $U^*\calK$, and $X$ with $\ti H=HU$ in \eqref{e-18-X} we complete the  proof of the theorem.
\end{proof}

Suppose that for the case of deficiency indices $(1,1)$ we have $\calK=\calK^*=\bar \calK=\kappa$\footnote{Throughout this paper  $\kappa$ will be
called the von Neumann  parameter.} and $U=1$. Then formula \eqref{e3-39-new} becomes
$$
H=\frac{i}{1-\kappa^2}[(1-\kappa)(1-\kappa)^{-1}-\kappa]=\frac{i}{1+\kappa}.
$$
Consequently, applying this value of $H$ to \eqref{e4-62} yields
\begin{equation}\label{e-15}
    S_\bA=\left(
            \begin{array}{cc}
              -\frac{i\kappa}{1+\kappa} & \frac{i}{1+\kappa} \\
              \frac{i\kappa^2}{1+\kappa}-i\kappa & i-\frac{i\kappa}{1+\kappa} \\
            \end{array}
          \right),\quad
S_{\bA^*}=\left(
            \begin{array}{cc}
              \frac{i\kappa}{1+\kappa}-i & -\frac{i\kappa^2}{1+\kappa}+i\kappa  \\
              -\frac{i}{1+\kappa} & \frac{i\kappa}{1+\kappa} \\
            \end{array}
          \right).
\end{equation}
Performing direct calculations gives
\begin{equation}\label{e-16}
    \frac{1}{2i}(S_{\bA}-S_{\bA^*})=\frac{1-\kappa}{2+2\kappa}\left(
            \begin{array}{cc}
              1 & 1 \\
              1 & 1 \\
            \end{array}
          \right).
\end{equation}
Using \eqref{e-16} with \eqref{e3-40} one obtains
\begin{equation}\label{e-17}
    \begin{aligned}
    \IM\bA&=\frac{1-\kappa}{2+2\kappa}\,\Big([(\cdot,\varphi)+(\cdot,\psi)]\varphi+ [(\cdot,\varphi)+(\cdot,\psi)]\psi\Big)\\
    &=\frac{1-\kappa}{2+2\kappa}\,(\cdot,\varphi+\psi)(\varphi+ \psi)\\
    &=(\cdot,\chi)\chi,
    \end{aligned}
\end{equation}
where
\begin{equation}\label{e-18}
    \chi=\sqrt{\frac{1-\kappa}{2+2\kappa}}\,(\varphi+ \psi)=\sqrt{\frac{1-\kappa}{1+\kappa}}\left(\frac{1}{\sqrt2}\,\varphi+ \frac{1}{\sqrt2}\,\psi\right).
\end{equation}

Consider a special case when $\kappa=0$. Then the corresponding ($*$)-extension $\bA_0$ is such that
\begin{equation}\label{e-19}
    \IM\bA_0=\frac{1}{2}(\cdot,\varphi+\psi)(\varphi+ \psi)=(\cdot,\chi_0)\chi_0,
    \end{equation}
where
\begin{equation}\label{e-20}
    \chi_0=\frac{1}{\sqrt2}\,\left(\varphi+\psi\right).
\end{equation}

\section{The Liv\v{s}ic function}\label{s4}

Suppose that $\dA$ is closed, prime,  densely defined symmetric operator with deficiency indices $(1,1)$.
In \cite[a part of Theorem 13]{L} (for a textbook exposition  see \cite{AG93}) M. Liv\v{s}ic suggested to call the function
\begin{equation}\label{charsym}
s(z)=\frac{z-i}{z+i}\cdot \frac{(g_z, g_-)}{(g_z, g_+)}, \quad z\in \bbC_+,
\end{equation}
{\it the characteristic function} of the symmetric operator $\dot A$. Here $g_\pm\in \Ker( \dA^*\mp iI)$ are normalized appropriately chosen deficiency elements and
$ g_z\ne 0$ are arbitrary  deficiency elements of the symmetric operators $\dot A$.
The Liv\v{s}ic result identifies the function $s(z)$ (modulo $z$-independent unimodular factor) with
a complete unitary invariant  of a prime  symmetric operator with deficiency indices $(1,1)$
that  determines the operator uniquely up to unitary equivalence.
He also gave  the following criterion \cite[Theorem 15]{L} (also see \cite{AG93})
 for a contractive analytic mapping
 from the upper half-plane  $\bbC_+$ to the unit disk
$\bbD$ to  be the characteristic function of a densely defined
symmetric operator  with deficiency indices $(1,1)$.

\begin{theorem}[\cite{L}]\label{thm12}\label{nach1} For an
 analytic mapping $s$ from the upper half-plane to the unit disk
to be the characteristic function of a densely defined
symmetric operator  with deficiency indices $(1,1)$
it is necessary and sufficient that
\begin{equation}\label{vsea0}
s(i)=0\quad \text{and}\quad \lim_{z\to \infty}
z(s(z)-e^{2i\alpha})=\infty \quad \text{for all} \quad  \alpha\in
[0, \pi),
\end{equation}
$$
0< \varepsilon \le \text{arg} (z)\le \pi -\varepsilon.
$$
\end{theorem}
\noindent
The \textbf{Liv\v{s}ic class} of functions described by Theorem \ref{thm12} will be denoted by $\sL$.

In the same article, Liv\v{s}ic put forward a concept of a characteristic function of a quasi-self-adjoint dissipative extension
of a symmetric operator with deficiency indices $(1,1)$.

Let us recall Liv\v{s}ic's construction. Suppose that $\dot A$ is a symmetric operator with deficiency indices $(1,1)$
and that $g_\pm$ are its normalized deficiency elements,
$$
g_\pm \in \Ker (\dA^*\mp i I), \quad \|g_\pm\|=1.
$$
Suppose that $T \ne (T )^*$ is a  maximal dissipative extension of $\dot A$,
$$\Im(T f,f)\ge 0, \quad f\in \dom(T )
.$$
Since $\dot A$ is symmetric, its dissipative extension $T $
is automatically quasi-self-adjoint \cite{ABT}, \cite{St68},
that  is,
$$
\dot A \subset T \subset \dA^*,
$$
and hence, according to \eqref{e4-60} with $\calK=\kappa$,
\begin{equation}\label{parpar}
g_+-\kappa g_-\in \dom
 (T )\quad \text{for some }
|\kappa|<1.
\end{equation}
Based on the  parametrization \eqref{parpar} of the domain of the
 extension $T $, Liv\v{s}ic suggested to call  the M\"obius transformation
\begin{equation}\label{ch12}
S(z)=\frac{s(z)-\kappa} {\overline{ \kappa }\,s(z)-1}, \quad z\in \bbC_+,
\end{equation}
where $s$ is given  by \eqref{charsym}, the \textbf{characteristic function} of the dissipative extension $T$ \cite{L}.
All functions that satisfy \eqref{ch12} for some function $s(z)\in\sL$ will form the \textbf{Liv\v{s}ic class} $\sL_\kappa$. Clearly, $\sL_0=\sL$.

A culminating point of Liv\v{s}ic's considerations was the discovery that the characteristic function $S(z)$ (up to a unimodular factor) of a dissipative (maximal) extension $T$ of a densely defined prime symmetric operator $\dot A$ is a complete unitary invariant of T (see  \cite[the remaining  part of Theorem 13]{L}).

In 1965  Donoghue \cite{D}
 introduced a concept of the Weyl-Titchmarsh function $M(\dot A, A)$
associated with a pair $(\dot A, A)$
by
$$M(\dot A, A)(z)=
((Az+I)(A-zI)^{-1}g_+,g_+), \quad z\in \bbC_+,
$$
$$g_+\in \Ker( \dA^*-iI),\quad \|g_+\|=1,
$$where $\dot A $ is
 a symmetric operator with deficiency indices $(1,1)$,
$\text{def}(\dot A)=(1,1)$,
 and $A$ is
its self-adjoint extension.

Denote by $\mM$ the \textbf{Donoghue class} of all analytic mappings $M$ from $\bbC_+$ into itself  that admits the representation
 \begin{equation}\label{hernev}
M(z)=\int_\bbR \left
(\frac{1}{\lambda-z}-\frac{\lambda}{1+\lambda^2}\right )
d\mu,
\end{equation}
where $\mu$ is an  infinite Borel measure   and
\begin{equation}\label{e-32-norm-m}
\int_\bbR\frac{d\mu(\lambda)}{1+\lambda^2}=1\,,\quad\text{equivalently,}\quad M(i)=i.
\end{equation}

It is known  \cite{D}, \cite{GMT}, \cite{GT}, \cite {MT-S} that $M\in \mM$ if and only if $M$
can be realized  as the Weyl-Titchmarsh function $M(\dot A, A)$ associated with a pair $(\dot A, A)$.
  The Weyl-Titchmarsh function $M$ is a (complete) unitary invariant  of the pair of a  symmetric operator with deficiency indices $(1,1)$ and its self-adjoint extension
 and determines the pair of operators uniquely up to unitary equivalence.

 Liv\v{s}ic's definition  of a characteristic function
of a symmetric operator (see eq. \eqref{charsym})
has some ambiguity related to the
 choice of the deficiency elements $g_\pm$.
To avoid this ambiguity we proceed as follows.
Suppose that  $A$ is a self-adjoint extension of a symmetric operator
$\dot A$ with deficiency indices $(1,1)$. Let
$g_\pm$  be deficiency elements $g_\pm\in \Ker ((\dot A)^*\mp iI)$,
$\|g_+\|=1$. Assume, in addition, that
\begin{equation}\label{star}
g_+-g_-\in \dom( A).
\end{equation}
Following \cite{MT-S} we introduce   the \textit{Liv\v{s}ic function} $s(\dot A, A)$
 \emph{associated with the pair} $(\dot A, A)$ by
\begin{equation}\label{charf12}
s(z)=\frac{z-i}{z+i}\cdot \frac{(g_z, g_-)}{(g_z, g_+)}, \quad
z\in \bbC_+,
\end{equation}
where $0\ne g_z\in \Ker((\dot A)^*-zI )$ is   an arbitrary
(deficiency) element.

A standard relationship between the  Weyl-Titch\-marsh and the    Liv\v{s}ic functions associated with the  pair
$(\dA, A) $ was described in \cite {MT-S}. In particular, if we denote by $M=M(\dot A, A)$ and by $s=s(\dot A, A)$
  the Weyl-Titchmarsh function and the    Liv\v{s}ic function  associated with the pair $(\dot A, A)$, respectively,
then
\begin{equation}\label{blog}
s(z)=\frac{M(z)-i}{M(z)+i},\quad z\in \bbC_+.
\end{equation}

\begin{hypothesis}\label{setup} Suppose
that $\whA \ne\whA^*$  is  a maximal
dissipative extension of  a symmetric operator $\dot A$
 with deficiency indices $(1,1)$. Assume, in addition, that $A$ is a  self-adjoint (reference) extension of $\dot A$. Suppose,  that the
deficiency elements $g_\pm\in \Ker (\dA^*\mp iI)$ are
normalized, $\|g_\pm\|=1$, and chosen in such a way that
\begin{equation}\label{ddoomm14}g_+- g_-\in \dom ( A)\,\,\,\text{and}\,\,\,
g_+-\kappa g_-\in \dom (\whA )\,\,\,\text{for some }
\,\,\,|\kappa|<1.
\end{equation}
\end{hypothesis}

Under Hypothesis \ref{setup}, we  introduce  the characteristic
function $S=S( \dot A, \whA , A)$
associated with the triple of operators $( \dot A, \whA , A)$
as the M\"obius transformation
\begin{equation}\label{ch1}
S(z)=\frac{s(z)-\kappa} {\overline{ \kappa }\,s(z)-1}, \quad z\in \bbC_+,
\end{equation}
of the Liv\v{s}ic function $s=s(\dot A, A)$ associated with the pair
$(\dot A, A)$.

We remark that given  a triple  $( \dot A, \whA , A)$, one can always find a
basis $g_\pm$ in the deficiency subspace
 $\Ker (\dA^*-iI)\dot +\Ker (\dA^*+iI)$,
$$\|g_\pm\|=1, \quad g_\pm\in \Ker (\dA^*\mp iI),
$$ such that
$$
g_+-g_-\in \dom (A)
\quad \text{
and}\quad
g_+-\kappa g_-\in \dom (\whA ),
$$
and then, in this case,
\begin{equation}\label{assa}
\kappa =S( \dot A, \whA , A)(i).
\end{equation}

Our next goal is to provide a {\it functional model} of a prime
 dissipative triple\footnote{We call a triple $(\dot A, \whA , A)$
 a prime triple if $\dot A$ is a prime symmetric operator.} parameterized by the characteristic
function and obtained in \cite{MT-S}.

Given a contractive analytic map $S$,
\begin{equation}\label{chchch}
S(z)=\frac{s(z)-\kappa} {\overline{ \kappa }\,s(z)-1}, \quad z\in \bbC_+,
\end{equation}
where $|\kappa|<1$ and $s$ is
an  analytic, contractive function in $\bbC_+$
satisfying the Liv\v{s}ic criterion \eqref{vsea0}, we use \eqref{blog} to introduce the function
$$
M(z)=\frac1i\cdot\frac{s(z)+1}{s(z)-1},\quad z\in \bbC_+,
$$
so that
$$
M(z)=\int_\bbR \left
(\frac{1}{\lambda-z}-\frac{\lambda}{1+\lambda^2}\right )
d\mu(\lambda), \quad z\in \bbC_+,
$$
for some infinite Borel measure with
$$
\int_\bbR\frac{d\mu(\lambda)}{1+\lambda^2}=1.
$$

In the Hilbert space $L^2(\bbR;d\mu)$ introduce
  the multiplication (self-adjoint)
operator by the  independent variable $\cB$
 on
\begin{equation}\label{nacha1}
\dom(\cB)=\left \{f\in \,L^2(\bbR;d\mu) \,\bigg | \, \int_\bbR
\lambda^2 | f(\lambda)|^2d \mu(\lambda)<\infty \right \},
\end{equation} denote by  $\dot \cB$  its
 restriction
on
\begin{equation}\label{nacha2}
\dom(\dot \cB)=\left \{f\in \dom(\cB)\, \bigg | \, \int_\bbR
f(\lambda)d \mu(\lambda) =0\right \},
\end{equation}
and let
 $\whB $ be   the dissipative restriction of the operator  $(\dot \cB)^*$
 on
\begin{equation}\label{nacha3}
\dom(\whB )=\dom (\dot \cB)\dot +\linspan\left
\{\,\frac{1}{\cdot -i}- S(i)\frac{1}{\cdot +i}\right \}.
\end{equation}

We will refer to the triple  $(\dot \cB,   \whB ,\cB)$ as
{\it  the model
 triple } in the Hilbert space $L^2(\bbR;d\mu)$.

It was established in \cite{MT-S} that a triple $(\dot A,\whA ,A)$
with the characteristic function $S$
is unitarily equivalent to the model triple
$(\dot \cB,   \whB ,\cB)$
in the Hilbert space $L^2(\bbR;d\mu)$
whenever the underlying symmetric operator $\dot A$ is prime.
The triple $(\dot \cB,   \whB ,\cB)$ will therefore be called {\it the functional model} for $(\dA, T, A)$.

It was pointed out in \cite{MT-S}, if $\kappa=0$, the quasi-self-adjoint extension $\whA $
coincides with the restriction of the adjoint operator $(\dot A)^*$ on
$$
\dom(\whA )=\dom(\dot A)\dot + \Ker (\dA^*-iI).
$$
and the  prime triples $(\dot A, \whA , A)$ with
$\kappa=0$
 are  in a one-to-one correspondence with the set of prime symmetric operators.
In this case, the characteristic function $S$
and the Liv\v{s}ic function $s$ coincide (up to a sign),
$$
S(z)=-s(z), \quad z\in \bbC_+.
$$

For the resolvents of the model dissipative operator $\whB$
and the self-adjoint (reference)
operator $\cB$ from the model  triple $(\dot \cB, \whB , \cB)$
 one gets
 the following resolvent formula.
\begin{proposition}[\cite{MT-S}] \label{t-11}
Suppose that
$(\dot \cB, \whB , \cB)$ is the model triple in the Hilbert space
$L^2(\bbR;d\mu) $.
 Then the resolvent  of the model dissipative operator $\whB $  in
$L^2(\bbR;d\mu) $ has the form
$$
(\whB - zI )^{-1}=(\cB- zI )^{-1}-p(z)(\cdot\, ,
g_{\overline{z}})g_z , $$ with
$$
p(z)=\left (M(\dot \cB,
\cB)(z)+i\frac{\kappa+1}{\kappa-1}\right )^{-1},
\quad z\in\rho(\whB )\cap \rho(\cB).$$
Here $M(\dot \cB,
\cB)$ is the Weyl-Titchmarsh function associated with the pair
 $(\dot \cB, \cB)$ continued to the lower half-plane by the Schwarz reflection
principle,
and the deficiency elements $g_z$ are given by
$$
g_z(\lambda)=\frac{1}{\lambda-z}, \quad
\,\, \text{$\mu$-a.e. }.
$$
\end{proposition}

\section{Transfer function vs Liv\v{s}ic function}

In this section and below, without loss of generality,  we can assume that $\kappa$ is real and that $0\le\kappa<1$. Indeed, if $\kappa=|\kappa|e^{i\theta}$, then
change (the basis) $g_-$ to $e^{i\theta}g_-$ in the deficiency subspace  $\Ker (\dA^*+ i I)$, say.
Thus, for the remainder of this paper we  suppose that the von Neumann parameter $\kappa$ is real and $0\le\kappa<1$.

The theorem below is the principal result of the current paper.
\begin{theorem}\label{t-25}
Let
\begin{equation}\label{e-62}
\Theta= \begin{pmatrix} \bA&K&\ 1\cr \calH_+ \subset \calH \subset
\calH_-& &\dC\cr \end{pmatrix}
\end{equation}be an  L-system whose main operator $T$ and the quasi-kernel $\hat A$ of $\RE\bA$ satisfy Hypothesis \ref{setup} with the reference operator $A=\hat A$ and the von Neumann parameter $\kappa$. Then the transfer function $W_\Theta(z)$ and the characteristic
function $S(z)$ of the triple $(\dA, T,\hat A)$ are reciprocals of each other, i.e.,
\begin{equation}\label{e-60}
    W_{\Theta}(z)=\frac{1}{S(z)},\quad z\in\dC_+\cap\rho(T),
\end{equation}
 and $\frac{1}{W_\Theta(z)}\in\sL_\kappa$.
\end{theorem}
\begin{proof} We are going to break the proof into three major steps.

\subsection*{Step 1.}

Let us consider the  model triple $(\dot \cB, \whBo, \cB)$ developed in Section \ref{s4} and described via formulas \eqref{nacha1}-\eqref{nacha3} with $\kappa=0$.
Let $\dB_0\in[\calH_+,\calH_-]$ be a  $(*)$-extension of $\whBo$ such that $\RE\dB_0\supset \cB=\cB^*$.  Clearly, $\whBo\in\Lambda(\dot \cB)$ and $\cB$ is the quasi-kernel of $\RE\dB_0$. It was shown in \cite[Theorem 4.4.6]{ABT} that $\dB_0$ exists and unique. We also note that by the construction of the model triple the von Neumann parameter $\calK=\kappa$ that parameterizes $\whBo$ via \eqref{e4-60}  equals  zero, i.e., $\calK=\kappa=0$. At the same time the parameter $U$ that parameterizes the quasi-kernel $\cB$ of $\RE\dB_0$ is equal to 1, i.e., $U=1$. Consequently, we can use the derivations of the end of Section \ref{s3} on $\dB_0$, use formulas \eqref{e-19}, \eqref{e-20} to conclude that
 \begin{equation}\label{e-48}
    \IM\dB_0=(\cdot,\chi_0)\chi_0, \quad \chi_0=\frac{1}{\sqrt2}\,\left(\varphi+\psi\right)\in\calH_-,
    \end{equation}
where $\varphi\in\calH_-$ and $\psi\in\calH_-$ are basis vectors in $\calR^{-1}(\sN_i)$ and $\calR^{-1}(\sN_{-i})$, respectively. Now we can construct (see \cite{ABT}) an L-system of the form
\begin{equation}\label{e-50}
\Theta_0= \begin{pmatrix} \dB_0&K_0&\ 1\cr \calH_+ \subset \calH \subset
\calH_-& &\dC\cr \end{pmatrix}
\end{equation}
where $K_0 c=c\cdot \chi_0$, $K_0^*f=(f,\chi_0)$, $(f\in\calH_+)$. The transfer function of this L-system can be written (see \eqref{e6-3-3}, \eqref{e-49} and \cite{ABT}) as
\begin{equation}\label{e-51}
    W_{\Theta_0}(z)=1-2i((\dB_0-z I)^{-1}\chi_0,\chi_0),\quad z\in\rho(\whBo),
\end{equation}
and the impedance function is\footnote{Here and below when we write $(\cB-z I)^{-1}\chi_0$ for $\chi_0\in\calH_-$ we mean that the resolvent $(\cB-z I)^{-1}$ is considered as extended to $\calH_-$ (see \cite{ABT}).}
\begin{equation}\label{e-52}
     V_{\Theta_0}(z)=((\RE\dB_0-z I)^{-1}\chi_0,\chi_0)=((\cB-z I)^{-1}\chi_0,\chi_0),\quad z\in\dC_\pm.
\end{equation}

At this point we  apply Proposition \ref{t-11} and obtain the following resolvent formula
\begin{equation}\label{e-49}
    (\whBo-z I)^{-1}=( \cB-z I)^{-1}-\frac{1}{M(\dot \cB,\cB)(z)-i}(\cdot,g_{\bar z})g_z,\quad z\in\rho(\whBo)\cap\dC_\pm,
\end{equation}
where $g_z=1/(t-z)$ and $M(\dot \cB,\cB)(z)$ is the Weyl-Titchmarsh function associated with the pair $(\dot \cB,\cB)$.
Moreover,
\begin{equation*}
    \begin{aligned}
    W_{\Theta_0}(z)&=1-2i((\dB_0-z I)^{-1}\chi_0,\chi_0)\\
    &=1-2i((\whBo-z I)^{-1}\chi_0,\chi_0)\\
    &=1-2i\left[(( \cB-z I)^{-1}\chi_0,\chi_0)-\left(\frac{1}{M(\dot \cB,\cB)(z)-i}(\chi_0,g_{\bar z})g_z,\chi_0\right)\right].
    \end{aligned}
\end{equation*}
Without loss of generality we can assume that
\begin{equation}\label{e-47-1}
g_z=(\cB-z I)^{-1}\chi_0=(\RE\dB_0-z I)^{-1}\chi_0=\frac{1}{t-z},\quad z\in\dC_\pm.
\end{equation}
Indeed, clearly $(\RE\dB_0-z I)^{-1}\chi_0\in\sN_z$, where $\sN_z$ is the deficiency subspace of $\dot \cB$,  and thus
$$
(\RE\dB_0-z I)^{-1}\chi_0=\frac{\xi}{t-z},\quad z\in\dC_\pm,
$$
for some $\xi\in\dC$. Let us show that $|\xi|=1$. For the impedance function $V_{\Theta_0}(z)$ in \eqref{e-52} we have
\begin{equation}\label{e-50-1}
    \begin{aligned}
    \IM V_{\Theta_0}(z)&=\frac{1}{2i}\left[((\RE\dB_0-z I)^{-1}\chi_0,\chi_0)-((\RE\dB_0-\bar z I)^{-1}\chi_0,\chi_0)\right]\\
    &=\frac{1}{2i}\left[(z-\bar z)( (\RE\dB_0-z I)^{-1}(\RE\dB_0-\bar z I)^{-1} \chi_0,\chi_0)\right] \\
    &=\IM z((\RE\dB_0-\bar z I)^{-1} \chi_0,(\RE\dB_0-\bar z I)^{-1} \chi_0 )\\
    &=\IM z\left(\frac{\xi}{t-\bar z},\frac{\xi}{t-\bar z}  \right)_{L^2(\bbR;d\mu)}=(\IM z) |\xi|^2\int_\bbR \frac{d\mu}{|t-z|^2}.
    \end{aligned}
\end{equation}

On the other hand, we know \cite{ABT} that $V_{\Theta_0}(z)$ is a Herglotz-Nevanlinna function that has integral representation
$$
V_{\Theta_0}(z)=Q+\int_\bbR \left (\frac{1}{t-z}-\frac{t}{1+t^2}\right )d\mu,\quad Q=\bar Q.
$$
Using the above representation, the property $\overline{V_{\Theta_0}}(z)=V_{\Theta_0}(\bar z)$, and straightforward calculations we find that
\begin{equation}\label{e-49-1}
   \IM V_{\Theta_0}(z)=(\IM z) \int_\bbR \frac{d\mu}{|t-z|^2}.
\end{equation}
Considering that $\int_\bbR \frac{d\mu}{|t-z|^2}>0$, we compare \eqref{e-50-1} with \eqref{e-49-1} and conclude that $|\xi|=1$. Since $|\xi|=1$, $\bar\xi$ can be scaled into $\chi_0$ and we obtain  \eqref{e-47-1}.

Taking into account \eqref{e-47-1}  and denoting $M_0=M(\dot \cB,\cB)(z)$ for the sake of simplicity, we continue
$$
\begin{aligned}
W_{\Theta_0}(z)&=1-2i\left(V_{\Theta_0}(z)-\frac{1}{M_0-i}V_{\Theta_0}^2(z)\right)\\
&=1-2i\left(i\frac{W_{\Theta_0}(z)-1}{W_{\Theta_0}(z)+1}+\frac{1}{M_0-i}\left(\frac{W_{\Theta_0}(z)-1}{W_{\Theta_0}(z)+1}\right)^2\right).
\end{aligned}
$$
Thus,
$$
W_{\Theta_0}(z)-1=2\frac{W_{\Theta_0}(z)-1}{W_{\Theta_0}(z)+1}-\frac{2i}{M_0-i}\left(\frac{W_{\Theta_0}(z)-1}{W_{\Theta_0}(z)+1}\right)^2,
$$
or
$$
1=\frac{2}{W_{\Theta_0}(z)+1}-\frac{2i}{M_0-i}\cdot\frac{W_{\Theta_0}(z)-1}{(W_{\Theta_0}(z)+1)^2}.
$$
Solving this equation for $W_{\Theta_0}(z)+1$ yields
\begin{equation}\label{e-53}
    W_{\Theta_0}(z)+1=\frac{(M_0-2i)\pm M_0}{M_0-i}.
\end{equation}
Assume that $M_0(z)\ne i$ for $z\in\dC_+$ and consider the  two outcomes for  formula \eqref{e-53}. First case leads to $W_{\Theta_0}(z)+1=2$ or $W_{\Theta_0}(z)=1$ which is impossible because it would lead (via \eqref{e6-3-6}) to $V_{\Theta_0}(z)=0$ that contradicts \eqref{e-49-1}. The second case is
$$
W_{\Theta_0}(z)+1=-\frac{2i}{M_0-i},
$$
leading to (see \eqref{blog})
$$
W_{\Theta_0}(z)=-\frac{2i}{M_0-i}-1=-\frac{M_0+i}{M_0-i}=-\frac{1}{s(z)},\quad z\in\dC_+\cap\rho(\whBo),
$$
where $s(z)$ is the Liv\v{s}ic function  associated with the pair $(\dot \cB, \cB)$. As we mentioned in Section \ref{s3}, in the case when $\kappa=0$ the characteristic function $S$
and the Liv\v{s}ic function $s$ coincide (up to a sign), or $S(z)=-s(z)$. Hence,
\begin{equation}\label{e-54}
    W_{\Theta_0}(z)=-\frac{1}{s(z)}=\frac{1}{S(z)},\quad z\in\dC_+\cap\rho(\whBo),
\end{equation}
where $S(z)$ is the characteristic function of the model triple $(\dot \cB, \whBo, \cB)$.

In the case when  $M_0(z)= i$ for all $z\in\dC_+$, formula \eqref{blog} would imply that $s(z)\equiv0$ in the upper half-plane. Then, as it was shown in \cite[Lemma 5.1]{MT-S}, all the points $z\in\dC_+$ are eigenvalues for $\whBo$ and the function $W_{\Theta_0}(z)$ is simply undefined in $\dC$ making \eqref{e-53} irrelevant.

As we mentioned  above, if $M_0(z)= i$ for all $z\in\dC_+$, the function $W_{\Theta_0}(z)$ is   ill-defined and \eqref{e-53} does not make sense  in $\dC_+$. One can, however, in this case re-write \eqref{e-53} in $\dC_-$. Using the symmetry of $M_0(z)$ we get that $M_0(z)= -i$ for all $z\in\dC_-$. Then \eqref{e-53} yields that $W_{\Theta_0}(z)=0$. On the other hand, \eqref{blog} extended to $\dC_-$ in this case implies that $s(z)=\infty$ for all $z\in\dC_-$ and hence \eqref{e-54} still formally holds true here for $z\in\dC_+\cap\rho(\whBo)$.

Let us also make one more observation. Using formulas \eqref{e6-3-6} and \eqref{e-54} yields
$$
W_{\Theta_0}(z)=\frac{1-iV_{\Theta_0}(z)}{1+iV_{\Theta_0}(z)}=-\frac{V_{\Theta_0}(z)+i}{V_{\Theta_0}(z)-i}=-\frac{M_0(z)+i}{M_0(z)-i},
$$
and hence
\begin{equation}\label{e-55}
    V_{\Theta_0}(z)=M_0(z),\quad z\in\dC_+.
\end{equation}

\subsection*{Step 2.}

Now we are ready to treat the case when $\kappa=\bar\kappa\ne0$.
Assume Hypothesis \ref{setup} and consider  the  model triple $(\dot \cB, \whB, \cB)$  described by formulas \eqref{nacha1}-\eqref{nacha3} with some  $\kappa$, $0\le\kappa<1$. Let $\dB\in[\calH_+,\calH_-]$ be a  $(*)$-extension of $\whB $ such that $\RE\dB\supset \cB=\cB^*$. Below we describe the construction of $\dB$. Equation \eqref{ddoomm14} of Hypothesis \ref{setup} implies that
$$
g_+- g_-\in \dom(\cB) \textrm{\quad or\quad }g_++ (-g_-)\in \dom(\cB),
$$
and
$$
g_+-\kappa g_-\in \dom (\whB )\textrm{\quad or\quad }g_++\kappa (-g_-)\in \dom (\whB ).
$$
Thus the von Neumann parameter $\calK$ that parameterizes $\whB $ via \eqref{e4-60} is $\kappa $ but the basis vector in $\sN_{-i}$ is $-g_-$. Consequently, $\calR^{-1}g_+=\varphi$ and $\calR^{-1}(-g_-)=-\psi$. Using \eqref{e-17} and \eqref{e-18} and replacing $\psi$ with $-\psi$, one obtains
\begin{equation}\label{e-57'}
       \IM\dB=(\cdot,\chi)\chi,\quad \chi=\sqrt{\frac{1-\kappa}{1+\kappa}}\left(\frac{1}{\sqrt2}\,\varphi- \frac{1}{\sqrt2}\,\psi\right).
    \end{equation}
We notice that if we followed the same basis pattern for the ($*$)-extension $\dB_0$ (when $\kappa=0$) then  \eqref{e-48} would become slightly modified as follows
\begin{equation}\label{e-58}
    \IM\dB_0=(\cdot,\chi_0)\chi_0, \quad \chi_0=\frac{1}{\sqrt2}\,\left(\varphi-\psi\right).
    \end{equation}
As before we use $\dB$ to construct  a model L-system of the form
\begin{equation}\label{e-59}
\Theta'= \begin{pmatrix} \dB&K'&\ 1\cr \calH_+ \subset \calH \subset
\calH_-& &\dC\cr \end{pmatrix},
\end{equation}
where $K' c=c\cdot \chi$, $K'^*f=(f,\chi)$, $(f\in\calH_+)$.

The impedance function of $\Theta'$   is
\begin{equation}\label{e-vtheta}
\begin{aligned}
& V_{\Theta'}(z)=((\RE\dB-z I)^{-1}\chi,\chi)=((\cB-z I)^{-1}\chi,\chi)\\
 &=\left((\cB-z I)^{-1}\sqrt{\frac{1-\kappa}{1+\kappa}}\left(\frac{1}{\sqrt2}\,\varphi- \frac{1}{\sqrt2}\,\psi\right),\sqrt{\frac{1-\kappa}{1+\kappa}}\left(\frac{1}{\sqrt2}\,\varphi- \frac{1}{\sqrt2}\,\psi\right)\right)\\
 &=\frac{1-\kappa}{1+\kappa}((\cB-z I)^{-1}\chi_0,\chi_0)=\frac{1-\kappa}{1+\kappa}\,V_{\Theta_0}(z)=\frac{1-\kappa}{1+\kappa}\,M_0(z),\quad z\in\dC_+.
 \end{aligned}
\end{equation}
Here we used relations \eqref{e-55} and \eqref{e-58}. On the other hand, using \eqref{ch1}, \eqref{e-54}, and \eqref{e-55} yields
$$
\begin{aligned}
 S(z)&=\frac{s(z)-\kappa} {{ \kappa }\,s(z)-1}=\frac{\frac{M_0-i}{M_0+i}-\kappa} {{ \kappa }\,\frac{M_0-i}{M_0+i}-1}=\frac{(1-\kappa)M_0-i(\kappa+1)}{(\kappa-1)M_0-(\kappa+1)i}\\
 &=-\frac{\frac{1-\kappa}{1+\kappa}M_0-i}{\frac{1-\kappa}{1+\kappa}M_0+i}=-\frac{V_{\Theta}(z)-i}{V_{\Theta}(z)+i}=\frac{1}{W_{\Theta}(z)}.
 \end{aligned}
$$
Thus,
\begin{equation}\label{e-60-prime}
    W_{\Theta'}(z)=\frac{1}{S(z)},\quad z\in\dC_+\cap\rho(\whB),
\end{equation}
where $S(z)$ is the characteristic function of the model triple $(\dot \cB, \whB, \cB)$.

\subsection*{Step 3.} Now we are ready to treat the general case.
Let
\begin{equation*}
\Theta= \begin{pmatrix} \bA&K&\ 1\cr \calH_+ \subset \calH \subset
\calH_-& &\dC\cr \end{pmatrix}
\end{equation*}
be an  L-system from the statement of our theorem.
Without loss of generality we can consider our L-system $\Theta$ to be minimal. If it is not minimal, we can use its so called ``principal part", which is an L-system that has the same transfer and impedance functions (see \cite[Section 6.6]{ABT}). We use the von Neumann parameter $\kappa$ of $T$ and the conditions of Hypothesis \ref{setup} to construct a model system $\Theta'$ given
by \eqref{e-59}. By construction $W_\Theta(z)=W_{\Theta'(z)}$ and the characteristic functions of $(\dA, T,\hat A)$ and the model triple $(\dot \cB, \whB, \cB)$ coincide. The conclusion of
the theorem then follows from Step 2 and formula \eqref{e-60-prime}.
\end{proof}

\begin{corollary} \label{c-26}
If under conditions of Theorem \ref{t-25} we also have that the von Neumann parameter  $\kappa$  of  $T$  equals zero, then $W_\Theta(z)=-1/s(z)$, where $s(z)$ is the Liv\v{s}ic function  associated with the pair $(\dA, \hat A)$.
\end{corollary}

\begin{corollary} \label{c-11}
Let $\Theta$ be an arbitrary  L-system of the form \eqref{e-62}. Then the transfer function of $W_\Theta(z)$ and the characteristic
function $S(z)$ of a triple $(\dA, T,\hat A_1)$ satisfying Hypothesis \ref{setup} with reference operator $A=\hat A_1$ are related via
\begin{equation}\label{e-56-arb}
    W_{\Theta}(z)=\frac{\nu}{S(z)},\quad z\in\dC_+\cap\rho(T),
\end{equation}
where $\nu\in\dC$ and $|\nu|=1$.
\end{corollary}
\begin{proof}
The only difference between the L-system $\Theta$ here and the one described in Theorem \ref{t-25} is that the set of  conditions of Hypothesis \ref{setup} is satisfied for the latter. Moreover, there is an L-system $\Theta_1$ of the form \eqref{e-62} with the same main operator $T$ that complies with  Hypothesis \ref{setup}. Then according to the theorem about a constant $J$-unitary factor \cite[Theorem 8.2.1]{ABT}, \cite{ArTs03},   $W_{\Theta}(z)=\nu W_{\Theta_1}(z)$, where $\nu$ is a unimodular complex number. Applying Theorem \ref{t-25} to the L-system $\Theta_1$ yields $W_{\Theta_1}(z)=1/S(z)$, where $S(z)$ is the characteristic function of the triplet $(\dA, T,\hat A_1)$ and $\hat A_1$ is the quasi-kernel of the real part of the operator $\bA_1$ in $\Theta_1$.
Consequently,
$$
W_{\Theta}(z)=\nu W_{\Theta_1}(z)=\frac{\nu}{S(z)},$$
where $|\nu|=1$.
\end{proof}

\section{Impedance functions of the classes $\sM$ and $\sM_\kappa$}\label{s5}

We  say that an analytic function $V$ from $\bbC_+$ into itself belongs to the \textbf{generalized Donoghue class} $\sM_\kappa$, ($0\le\kappa<1$) if it admits the representation \eqref{hernev} with an  infinite Borel measure $\mu$ and
\begin{equation}\label{e-61-kappa}
\int_\bbR\frac{d\mu(\lambda)}{1+\lambda^2}=\frac{1-\kappa}{1+\kappa}\,,\quad\text{equivalently,}\quad V(i)=i\,\frac{1-\kappa}{1+\kappa}.
\end{equation}
Clearly, $\sM_0=\sM$.

We proceed by stating and proving the following important lemma.
\begin{lemma}\label{l-20}
Let $\Theta_\kappa$ of the form \eqref{e-62} be an  L-system whose main operator $T$ (with the von Neumann parameter $\kappa$, $0\le\kappa<1$) and the quasi-kernel $\hat A$ of $\RE\bA$ satisfy the conditions of Hypothesis \ref{setup} with the reference operator $A=\hat A$.
Then the impedance function $V_{\Theta_\kappa}(z)$ admits the representation
\begin{equation}\label{e-imp-m}
    V_{\Theta_\kappa}(z)=\frac{1-\kappa}{1+\kappa}\,V_{\Theta_0}(z),\quad z\in\dC_+,
\end{equation}
where $V_{\Theta_0}(z)$ is  the impedance function of an L-system $\Theta_0$ with the same set of conditions but with $\kappa_0=0$, where $\kappa_0$ is the von Neumann parameter of the main operator $T_0$ of $\Theta_0$.
\end{lemma}
\begin{proof}
 Once again we rely on our derivations above. We use the von Neumann parameter $\kappa$ of $T$ and the conditions of Hypothesis \ref{setup} to construct a model system $\Theta'$ given
by \eqref{e-59}. By construction $V_{\Theta_\kappa}(z)=V_{\Theta'}(z)$. Similarly, the impedance function $V_{\Theta_0}(z)$ coincides with the impedance function of a model system \eqref{e-50}.  The conclusion of the lemma then follows from \eqref{e-55} and \eqref{e-vtheta}.
\end{proof}

\begin{theorem}\label{t-21}
Let $\Theta$ of the form \eqref{e-62} be an L-system whose main operator $T$ has the von Neumann parameter $\kappa$, $0\le\kappa<1$. Then its impedance function $V_\Theta(z)$ belongs to the Donoghue class $\sM$ if and only if $\kappa=0$.
\end{theorem}
\begin{proof}
First of all, we note that in our system $\Theta$ the quasi-kernel $\hat A$ of $\RE\bA$ does not necessarily satisfy the conditions of Hypothesis \ref{setup}. However, if $\Theta_\kappa$ is a system from the statement of Lemma \ref{l-20} with the same $\kappa$ and Hypothesis \ref{setup} requirements, then
\begin{equation}\label{e-58-nu}
W_\Theta(z)=\nu W_{\Theta_\kappa}(z),
\end{equation}
where $\nu$ is a complex number such that $|\nu|=1$. This follows from the theorem about a constant $J$-unitary factor \cite[Theorem 8.2.1]{ABT}, \cite{ArTs03}.

To prove the Theorem in one direction we assume that $V_\Theta(z)\in\sM$ and $\kappa\ne0$.
We know that Theorem \ref{t-25} applies to the L-system $\Theta_\kappa$  and hence formula \eqref{e-60} takes place.
Combining \eqref{e-60} with \eqref{e-58-nu} and using the normalization condition \eqref{assa} we obtain
\begin{equation}\label{e-59-nu}
    W_\Theta(i)=\frac{\nu}{\kappa}.
\end{equation}
We also know that according to \cite[Theorem 6.4.3]{ABT} the impedance function $V_\Theta(z)$ admits the following integral representation
\begin{equation}\label{e-60-nu}
V_\Theta(z)=Q+\int_\bbR \left(\frac{1}{\lambda-z}-\frac{\lambda}{1+\lambda^2}\right)d\mu,
\end{equation}
where $Q$ is a real number and $\mu$ is an  infinite Borel measure   such that
$$
\int_\bbR\frac{d\mu(\lambda)}{1+\lambda^2}=L<\infty.
$$
It follows directly from \eqref{e-60-nu} that $V_\Theta(i)=Q+iL$. Therefore, applying \eqref{e6-3-6} directly to $W_\Theta(z)$ and using \eqref{e-59-nu} yields
$$
\begin{aligned}
W_\Theta(i)&=\frac{1-i V_\Theta(i)}{1+i V_\Theta(i)}=\frac{1-i(Q+iL)}{1+i(Q+iL)}=\frac{1+L-iQ}{1-L+iQ}=\frac{\nu}{\kappa}.
\end{aligned}
$$
Cross multiplying yields
\begin{equation}\label{e-67-bad}
\kappa+\kappa L-i\kappa Q=\nu-\nu L+i\nu Q.
\end{equation}
Solving this relation for $Q$ gives us
\begin{equation}\label{e-61-nu}
    Q=i\frac{\nu(1-L)-\kappa(1+L)}{\nu+\kappa}.
\end{equation}
Taking into account that  $\nu\bar \nu=1$ and recalling our agreement in Section \ref{s3} to consider real $\kappa$ only, we get
\begin{equation}\label{e-63-nu}
\bar Q=-i\frac{\bar\nu(1-L)-\kappa(1+L)}{\bar \nu+\kappa}=-i\frac{(1-L)-\kappa\nu(1+L)}{1+ \nu\kappa}.
\end{equation}
But $Q=\bar Q$ and hence equating \eqref{e-61-nu} and \eqref{e-63-nu} and solving for $L$ yields
\begin{equation}\label{e-62-nu}
    L=\frac{\nu-\kappa^2\nu}{(\nu+\kappa)(1+\kappa\nu)}.
\end{equation}
Clearly,  $V_\Theta(z)\in\sM$ if and only if $Q=0$ and $L=1$. Setting the right hand side of \eqref{e-62-nu} to $1$ and solving for $\kappa$ gives $\kappa=0$ or $\kappa=-(\nu^2+1)/(2\nu)$, but only $\kappa=0$ makes $Q=0$ in \eqref{e-61-nu}. Consequently, our assumption that $\kappa\ne0$  leads to a  contradiction. Therefore, $V_\Theta(z)\in\sM$ implies $\kappa=0$.

In order to prove the converse we assume that $\kappa=0$. Let $\Theta_0$ be the L-system $\Theta_\kappa$ described in the beginning of the proof with $\kappa=0$. Let also $\hat A_0$ be the reference operator in $\Theta_0$ that is the quasi-kernel of the real part of the state-space operator in $\Theta_0$. Then the fact that $S(\dA,T,\hat A_0)(z)=-s(\dA,\hat A_0)(z)$ for $\kappa=0$ (see Section \ref{s4}) and \eqref{blog} yield
\begin{equation*}
W_\Theta(z)=\nu W_{\Theta_0}(z)= \frac{\nu}{S(\dA,\hat A_0)(z)}= -\frac{\nu}{s(\dA,\hat A_0)(z)}= \frac{\nu(M(\dA,\hat A_0)(z)+i)}{i-M(\dA,\hat A_0)(z)}.
\end{equation*}
Moreover, applying \eqref{e6-3-6} to the above formula for $W_\Theta(z)$ we obtain
\begin{equation}\label{e-71-newww}
   V_\Theta(z)=i \frac{W_\Theta (z) - 1}{W_\Theta (z) + 1}=i \frac{\frac{\nu(M(\dA,\hat A_0)(z)+i)}{i-M(\dA,\hat A_0)(z)} - 1}{\frac{\nu(M(\dA,\hat A_0)(z)+i)}{i-M(\dA,\hat A_0)(z)} + 1}
   =i\frac{(1+\bar\nu)M(\dA,\hat A_0)(z)+(1-\bar\nu)i}{(1-\bar\nu)M(\dA,\hat A_0)(z)+(1+\bar\nu)i}.
\end{equation}
Substituting $z=-i$ to \eqref{e-71-newww} yields $V_\Theta(-i)=-i$ and thus, by symmetry property of $V_\Theta(z)$,  we have that $V_\Theta(i)=i$ and hence $V_\Theta(z)\in\sM$.



\end{proof}
Consider the L-system $\Theta$ of the form \eqref{e-62} that was used in the statement of Theorem \ref{t-21}. This L-system does not necessarily comply with the conditions of Hypothesis  \ref{setup} and hence the quasi-kernel $\hat A$ of $\RE\bA$ is parameterized via \eqref{DOMHAT} by some complex number  $U$, $|U|=1$. Then $U=e^{2i\beta}$, where $\beta\in[0,\pi)$. This representation allows us to introduce a one-parametric family of L-systems $\Theta_0(\beta)$ that all have $\kappa=0$. That is,
\begin{equation}\label{e-62-beta}
\Theta_0(\beta)= \begin{pmatrix} \bA_0(\beta)&K_0(\beta)&\ 1\cr \calH_+ \subset \calH \subset
\calH_-& &\dC\cr \end{pmatrix}.
\end{equation}
We note that $\Theta_0(\beta)$ satisfies the conditions of Hypothesis  \ref{setup} only for the case when $\beta=0$. Hence, the L-system $\Theta_0$ from Lemma \ref{l-20} can be written as $\Theta_0=\Theta_0(0)$ using \eqref{e-62-beta}.  Moreover, it directly follows from Theorem \ref{t-21} that all the impedance functions $V_{\Theta_0(\beta)}(z)$ belong to the Donoghue class $\sM$ regardless of the value of $\beta\in[0,\pi)$.

The next theorem gives criteria on when the impedance function of an L-system belongs to the generalized Donoghue class $\sM_\kappa$.

\begin{theorem}\label{t-22}
Let $\Theta_\kappa$, $0<\kappa<1$, of the form \eqref{e-62} be a minimal L-system with the main operator $T$ and the impedance function $V_{\Theta_\kappa}(z)$ which is not an identical constant in $\dC_+$.
Then  $V_{\Theta_\kappa}(z)$ belongs to the generalized Donoghue class $\sM_\kappa$ and \eqref{e-imp-m} holds if and only if  the triple $(\dot A, T, \hat A)$ satisfies Hypothesis \ref{setup} with $A=\hat A$, the quasi-kernel of $\RE\bA$.
\end{theorem}
\begin{proof}
We prove the necessity first. Suppose the triple $(\dot A, T, \hat A)$  in $\Theta$ satisfies the conditions of  Hypothesis \ref{setup}. Then, according to Lemma \ref{l-20}, formula \eqref{e-imp-m} holds and consequently $V_{\Theta_\kappa}(z)$ belongs to the generalized Donoghue class $\sM_\kappa$.

In order to prove the Theorem in the other direction we assume that $V_{\Theta_\kappa}(z)\in\sM_\kappa$ satisfies  equation \eqref{e-imp-m} for some L-system $\Theta_0$. Then according to Theorem \ref{t-21} $V_{\Theta_0}(z)$ belongs to the Donoghue class $\sM$. Clearly then \eqref{e-imp-m} implies that $V_{\Theta_\kappa}(z)$ has $Q=0$ in its integral representation \eqref{e-61-nu}. Moreover,
$$
V_{\Theta_\kappa}(i)=\frac{1-\kappa}{1+\kappa}\,V_{\Theta_0}(i)=i\frac{1-\kappa}{1+\kappa}\,=i\int_\bbR\frac{d\mu(\lambda)}{1+\lambda^2},
$$
where $\mu(\lambda)$ is  the measure from the integral representation \eqref{e-61-nu} of $V_{\Theta_\kappa}(z)$. Thus,
$$
L=\int_\bbR\frac{d\mu(\lambda)}{1+\lambda^2}=\frac{1-\kappa}{1+\kappa}.
$$
Assume the contrary, i.e., suppose that the quasi-kernel $\hat A$ of $\RE\bA$ of $\Theta_\kappa$ does not satisfy the conditions of Hypothesis \ref{setup}. Then, consider another  L-system $\Theta'$ of the form \eqref{e-62} which is only different from $\Theta$ by that its quasi-kernel $\hat A'$ of $\RE\bA'$ satisfies the conditions of Hypothesis \ref{setup} for the same value of $\kappa$. Applying the theorem about a constant $J$-unitary factor \cite[Theorem 8.2.1]{ABT} then yields
\begin{equation*}\label{e-58-nu''}
W_{\Theta_\kappa}(z)=\nu W_{\Theta'}(z),
\end{equation*}
where $\nu$ is a complex number such that $|\nu|=1$. Our goal is to show that $\nu=1$. Since we know the values of $Q$ and $L$ in the integral representation \eqref{e-61-nu} of $V_{\Theta_\kappa}(z)$, we can use this information to find $\nu$ from \eqref{e-61-nu}. We have then
$$
0=i\frac{\nu(1-L)-\kappa(1+L)}{\nu+\kappa}, \quad\textrm{ where }\quad L=\frac{1-\kappa}{1+\kappa}.
$$
Consequently, $\nu(1-L)-\kappa(1+L)=0$ or
$$
\nu=\kappa\,\frac{1+L}{1-L}=\kappa\,\frac{1+\frac{1-\kappa}{1+\kappa}}{1-\frac{1-\kappa}{1+\kappa}}=\kappa\cdot\frac{2}{2\kappa}=1.
$$
Thus, $\nu=1$ and hence
\begin{equation}\label{e-64-nu}
W_{\Theta_\kappa}(z)= W_{\Theta'}(z).
\end{equation}
Our L-system $\Theta_\kappa$ is minimal and hence we can apply the Theorem on bi-unitary equivalence \cite[Theorem 6.6.10]{ABT} for L-systems $\Theta_\kappa$ and $\Theta'$ and obtain that the pairs $(\dA, \hat A)$ and $(\dA, \hat A')$ are unitarily equivalent. Consequently, the  Weyl-Titchmarsh functions $M(\dA,\hat A)$ and $M(\dA,\hat A' )$ coincide. At the same time, both $\hat A$ and $\hat A'$  are self-adjoint extensions of the symmetric operator $\dA$ giving us the following relation between $M(\dA,\hat A)$ and $M(\dA,\hat A' )$ (see \cite[Subsection 2.2]{MT10})
\begin{equation}\label{transm}
M(\dA, \hat A)=\frac{\cos \alpha \, M(\dot A, \hat A')-\sin \alpha}{
\cos\alpha +\sin \alpha \,M(\dot A, \hat A')},
\quad\textrm{ for some }\alpha \in [0,\pi).
\end{equation}
  Using  $M(\dA, \hat A')(z)=M(\dA, \hat A)(z)$ for $z\in\dC_+$ on \eqref{transm} and solving for $M(\dot A, \hat A)(z)$ gives us that either $\alpha=0$ or $M(\dot A, \hat A)(z)=i$ for all $z\in\dC_+$. The former case of $\alpha=0$ gives $\hat A=\hat A'$, and thus $\hat A$  satisfies the conditions of Hypothesis \ref{setup} which contradicts our assumption. The latter case  would imply (via \eqref{blog}) that $s(z)=s(\dA,\hat A)(z)\equiv0$ and consequently $S(z)=S(\dA,\hat A,T)(z)\equiv \kappa$ in the upper half-plane. Then \eqref{e-60} and \eqref{e-56-arb} yield $W_{\Theta_\kappa}(z)=\theta/\kappa$ for some $\theta$ such that $|\theta|=1$ and hence
\begin{equation}\label{e-72-const}
V_{\Theta_\kappa}(z)=i\frac{\theta/\kappa-1}{\theta/\kappa+1}=i\frac{\theta-\kappa}{\theta+\kappa},\quad z\in\dC_+.
\end{equation}
Thus, in particular,
$$
V_{\Theta_\kappa}(i)=i\frac{\theta-\kappa}{\theta+\kappa}.
$$
On the other hand, we know that $V_{\Theta_\kappa}(z)$ satisfies  equation \eqref{e-imp-m} and hence (taking into account that  $V_{\Theta_0}(i)=i$),  plugging $z=i$ in \eqref{e-imp-m} gives
$$
V_{\Theta_\kappa}(i)=i\frac{1-\kappa}{1+\kappa}.
$$
Combining the two equations above we get  $\theta=1$. Therefore, \eqref{e-72-const} yields
 \begin{equation}\label{e-73-id-const}
V_{\Theta_\kappa}(z)=i\frac{1-\kappa}{1+\kappa},\quad z\in\dC_+,
\end{equation}
 which brings us back to  a contradiction with a condition of the Theorem that $V_{\Theta_\kappa}(z)$  is not an identical constant. Consequently, $\alpha=0$ is the only feasible choice and hence $\hat A=\hat A'$ implying that $\hat A$  satisfies the conditions of Hypothesis \ref{setup}.
\end{proof}
\begin{remark}
Let us consider the case when the condition of $V_{\Theta_\kappa}(z)$ not being an identical constant in $\dC_+$ is omitted in the statement of Theorem \ref{t-22}. Then, as we have shown in the proof of the theorem, $V_{\Theta_\kappa}(z)$ may take a form \eqref{e-73-id-const}. We will show that in this case the L-system $\Theta$ from the statement of Theorem \ref{t-22} is bi-unitarily equivalent to an L-system $\Theta'$ that satisfies the conditions of Hypothesis \ref{setup}.

Let $V_{\Theta_\kappa}(z)$ from Theorem \ref{t-22} takes a form \eqref{e-73-id-const}. Let also $\mu(\lambda)$ be a Borel measure on $\Bbb R$ given by the simple formula
\begin{equation}\label{e-75-mu}
    \mu(\lambda)=\frac{\lambda}{\pi},\quad \lambda\in\Bbb R,
\end{equation}
and let $V_0(z)$ be a function with integral representation \eqref{hernev} with the measure $\mu$, i.e.,
\begin{equation*}\label{e-76-V0}
V_0(z)=\int_\bbR \left
(\frac{1}{\lambda-z}-\frac{\lambda}{1+\lambda^2}\right )
d\mu.
\end{equation*}
Then by direct calculations one immediately finds that $V_0(i)=i$ and that $V_0(z_1)-V_0(z_2)=0$ for any $z_1\ne z_2$ in $\dC_+$. Therefore, $V_0(z)\equiv i$ in $\dC_+$ and hence using \eqref{e-73-id-const}   we obtain \eqref{e-imp-m} or
 \begin{equation}\label{e-73-id-const-1}
V_{\Theta_\kappa}(z)=i\frac{1-\kappa}{1+\kappa}=\frac{1-\kappa}{1+\kappa}\,V_0(z),\quad z\in\dC_+.
\end{equation}
Let us construct a model triple $(\dot \cB,   \whB ,\cB)$ defined by \eqref{nacha1}--\eqref{nacha3} in the Hilbert space $L^2(\bbR;d\mu)$ using the measure $\mu$ from \eqref{e-75-mu} and our value of $\kappa$. Using the formula for the deficiency elements $g_z(\lambda)$ of $\dot B$ (see Proposition \ref{t-11}) and the definition of $s(\dot B, \cB)(z)$ in \eqref{charf12} we evaluate that $s(\dot B, \cB)(z)\equiv0$ in $\dC_+$. Then, \eqref{chchch} yields $S(\dot \cB,   \whB ,\cB)(z)\equiv \kappa$ in $\dC_+$. Moreover, applying Proposition \ref{t-11} to the operator $\whB$ in our triple we obtain
\begin{equation}\label{e-77-resolv}
(\whB - zI )^{-1}=(\cB- zI )^{-1}+i\left(\frac{\kappa-1}{2\kappa}\right)(\cdot\, ,g_{\overline{z}})g_z.
\end{equation}
Let us now follow Step 2 of the proof of Theorem \ref{t-25} to construct a model L-system $\Theta'$ of the form \eqref{e-59} corresponding to our model triple $(\dot \cB,   \whB ,\cB)$. Note, that this L-system $\Theta'$ is minimal by construction, its main operator $\whB$ has regular points in $\dC_+$ due to \eqref{e-77-resolv}, and, according to \eqref{e-60}, $W_{\Theta'}(z)\equiv 1/\kappa$. But formulas \eqref{e6-3-6} yield that in the case under consideration  $W_{\Theta_\kappa}(z)\equiv 1/\kappa$. Therefore $W_{\Theta_\kappa}(z)=W_{\Theta'}(z)$ and we can (taking into account the properties of $\Theta'$ we mentioned) apply the Theorem on bi-unitary equivalence \cite[Theorem 6.6.10]{ABT} for L-systems $\Theta_\kappa$ and $\Theta'$. Thus we have successfully constructed an L-system $\Theta'$ that  is bi-unitarily equivalent to the L-system $\Theta_\kappa$ and satisfies the conditions of Hypothesis \ref{setup}.
\end{remark}
Using similar reasoning as above we introduce another one parametric family of L-systems
\begin{equation}\label{e-62-beta-kappa}
\Theta_\kappa(\beta)= \begin{pmatrix} \bA_\kappa(\beta)&K_\kappa(\beta)&\ 1\cr \calH_+ \subset \calH \subset
\calH_-& &\dC\cr \end{pmatrix},
\end{equation}
which is different from the family in \eqref{e-62-beta} by the fact that all the members of the family have the same operator $T$ with the fixed von Neumann parameter $\kappa\ne0$. It easily follows from Theorem \ref{t-22} that for all $\beta\in[0,\pi)$ there is  only one non-constant in $\dC_+$ impedance function $V_{\Theta_\kappa(\beta)}(z)$ that belongs to the class $\sM_\kappa$. This happens when $\beta=0$ and consequently the L-system $\Theta_\kappa(0)$ complies with the conditions of Hypothesis \ref{setup}. The results of Theorems \ref{t-21}  and \ref{t-22} can be illustrated with the help of Figure \ref{fig-1} describing the parametric region for the family of L-systems $\Theta(\beta)$.  When $\kappa=0$ and $\beta$ changes from $0$ to $\pi$,  every point on the unit circle with cylindrical coordinates $(1,\beta,0)$, $\beta\in[0,\pi)$ describes an L-system $\Theta_0(\beta)$ and  Theorem  \ref{t-21} guarantees that $V_{\Theta_0(\beta)}(z)$ belongs to the class $\sM$. On the other hand, for any $\kappa_0$ such that $0<\kappa_0<1$ we apply Theorem  \ref{t-22} to conclude that only the point $(1,0,\kappa_0)$ on the wall of the cylinder is responsible for an L-system  $\Theta_{\kappa_0}(0)$ such that $V_{\Theta_{\kappa_0}(0)}(z)$ belongs to the class $\sM_{\kappa_0}$.
\begin{figure}
  \begin{center}
  \includegraphics[width=70mm]{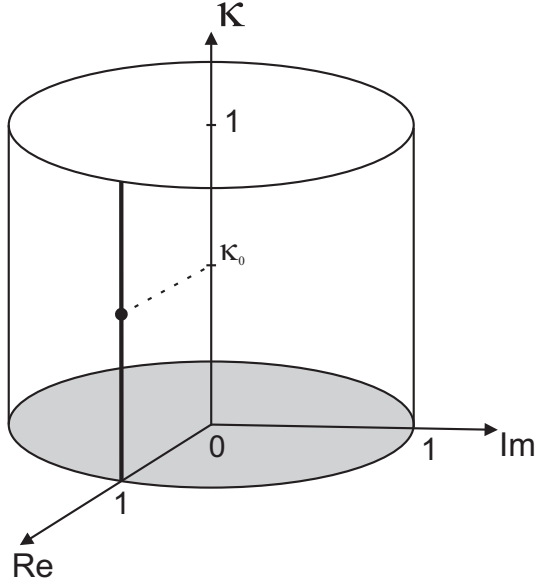}
  \caption{Parametric region $0\le\kappa<1$, $0\le\beta<\pi$}\label{fig-1}
  \end{center}
\end{figure}

\begin{theorem}\label{t-14}
Let  $V(z)$ belong to the generalized Donoghue class $\sM_\kappa$, $0\le\kappa<1$. Then $V(z)$ can be realized as the impedance function $V_{\Theta_\kappa}(z)$ of an L-system $\Theta_\kappa$
 of the form \eqref{e-62} with  the triple $(\dot A, T, \hat A)$ that satisfies Hypothesis \ref{setup} with $A=\hat A$, the quasi-kernel of $\RE\bA$. Moreover,
 \begin{equation}\label{e-73-new}
    V(z)=V_{\Theta_\kappa}(z)= \frac{1-\kappa}{1+\kappa}\,M(\dA, \hat A)(z),\quad z\in\dC_+,
 \end{equation}
 where $M(\dA, \hat A)(z)$ is the Weyl-Titchmarsh function associated with the pair $(\dA, \hat A)$.
\end{theorem}
\begin{proof}
Since $V(z)\in\sM_\kappa$, then it admits the integral representation \eqref{hernev} with normalization condition \eqref{e-61-kappa} on the measure $\mu$. Set
$$
c=\frac{1+\kappa}{1-\kappa}.
$$
It follows directly from definitions of classes $\sM$ and $\sM_\kappa$ that the function $cV\in\sM$ and thus has the integral representation \eqref{hernev} with the measure $\mu_0=c\mu$ and normalization condition \eqref{e-32-norm-m} on the measure $\mu_0$.
We use the measure $\mu_0$  to construct a model triple $(\dot \cB,   \whBo ,\cB)$ described by \eqref{nacha1}-\eqref{nacha3} with  $S(i)=0$. Note that the model triple $(\dot \cB,   \whBo ,\cB)$ satisfies Hypothesis \ref{setup}. Then we follow Step 1 of the proof of Theorem \ref{t-25} to build an L-system $\Theta_0$ given by \eqref{e-50}. According to \eqref{e-55}
$V_{\Theta_0}(z)=M(\dot\cB, \cB)(z)$. On the other hand, since $M(\dot\cB, \cB)(z)$ is the Weyl-Titchmarsh function associated with the pair $(\dot\cB, \cB)$, then it also admits a representation
$$
M(\dot\cB, \cB)(z)=\int_\bbR \left
(\frac{1}{\lambda-z}-\frac{\lambda}{1+\lambda^2}\right )
d\mu_0,\quad z\in\dC_+,
$$
with the same measure $\mu_0$ as in the representation for $cV$. Therefore,
$$
cV(z)=M(\dot\cB, \cB)(z)=V_{\Theta_0}(z),\quad z\in\dC_+,
$$
or $V(z)=(1/c)V_{\Theta_0}(z)$. Then we proceed with Step 2 of the proof of Theorem \ref{t-25} to construct an L-system $\Theta'$ given by \eqref{e-59}.
It is  shown in \eqref{e-vtheta} that
\begin{equation}\label{e-74-now}
V_{\Theta'}(z)= \frac{1-\kappa}{1+\kappa}\,M(\dot\cB, \cB)(z),\quad z\in\dC_+,
\end{equation}
and hence
$$
V_{\Theta'}(z)= \frac{1-\kappa}{1+\kappa}\,M(\dot\cB, \cB)(z)=\frac{1-\kappa}{1+\kappa}\,cV(z)=V(z).
$$
Therefore, we have constructed an L-system $\Theta_\kappa=\Theta'$ such that $V(z)=V_{\Theta_\kappa}(z)$.
The remaining part of \eqref{e-73-new} follows from \eqref{e-74-now}.
\end{proof}


\section{Examples}

\subsection*{Example 1}\label{ex-1}
Following \cite{AG93} we consider the prime symmetric operator
\begin{equation}\label{e-56}
    \begin{aligned}
\dA x&=i\frac{dx}{dt},\;
\dom(\dA)=\left\{x(t)\,\Big|\,x(t) -\text{abs. cont.}, x'(t)\in L^2_{[0,\ell]},\, x(0)=x(\ell)=0\right\}.\\
    \end{aligned}
\end{equation}
Its  (normalized) deficiency vectors of $\dA$ are
\begin{equation}\label{e-57}
g_+=\frac{\sqrt2}{\sqrt{e^{2\ell}-1}}e^t\in \sN_i,\qquad g_-=\frac{\sqrt2}{\sqrt{1-e^{-2\ell}}}e^{-t}\in \sN_{-i}.
\end{equation}
If we set $C=\frac{\sqrt2}{\sqrt{e^{2\ell}-1}}$, then \eqref{e-57} can be re-written as
$$
g_+=Ce^t,\quad g_-=Ce^\ell e^{-t}.
$$
Let
\begin{equation}\label{e-58'}
    \begin{aligned}
A x&=i\frac{dx}{dt},\;
\dom(A)=\left\{x(t)\,\Big|\,x(t) -\text{abs. cont.}, x'(t)\in L^2_{[0,\ell]},\, x(0)=-x(\ell)\right\}.\\
    \end{aligned}
\end{equation}
be a self-adjoint extension of $\dA$. Clearly, $g_+(0)-g_-(0)=C-Ce^\ell$ and $g_+(\ell)-g_-(\ell)=Ce^\ell-C$ and hence \eqref{star} is satisfied, i.e., $g_+-g_-\in\dom(A)$.

 Then the Liv\v{s}ic characteristic function $s(z)$ for the pair $(\dA,A)$ ha stye form  (see \cite{AG93})
\begin{equation}\label{e1-1}
s(z)=\frac{e^\ell-e^{-i\ell z}}{1-e^\ell e^{-i\ell z}}.
\end{equation}
We introduce the operator
\begin{equation}\label{e1-3}
    \begin{aligned}
T x&=i\frac{dx}{dt},\;
\dom(T)=\left\{x(t)\,\Big|\,x(t)-\text{abs. cont.}, x'(t)\in L^2_{[0,\ell]},\,  x(0)=0\right\}.
    \end{aligned}
\end{equation}
By construction, $T$ is a dissipative extension of $\dA$ parameterized by a von Neumann parameter $\kappa$. To find $\kappa$ we use \eqref{e-57} with \eqref{parpar} to obtain
\begin{equation}\label{e1-4}
x(t)=Ce^{t}-\kappa\, Ce^\ell e^{-t}\in\dom(T),\quad x(0)=0,
\end{equation}
yielding
\begin{equation}\label{e-62-1}
\kappa=e^{-\ell}.
\end{equation}
Obviously, the triple of operators $( \dot A, T, A)$ satisfy the conditions of Hypothesis \ref{setup} since $|\kappa|=e^{-\ell}<1$.
Therefore, we can use \eqref{ch1} to write out the characteristic function $S(z)$ for the triple $(\dA, T, A)$
\begin{equation}\label{e1-2}
S(z)=\frac{s(z)-\kappa}{\bar \kappa s(z)-1}=\frac{e^\ell-\kappa+e^{-i\ell z}(\kappa e^\ell-1)}{\bar\kappa e^\ell-1+e^{-i\ell z}(e^\ell-\bar\kappa)},
\end{equation}
and apply the value of $\kappa=e^{-\ell}$ to get
\begin{equation}\label{e-64}
S(z)=e^{i\ell z}.
\end{equation}

Now we shall use the triple $(\dA, T, A)$ for an L-system $\Theta$ that we about to construct. First, we note that by the direct check one gets
\begin{equation}\label{e1-3star}
    \begin{aligned}
T^* x&=i\frac{dx}{dt},\;
\dom(T)=\left\{x(t)\,\Big|\,x(t)-\text{abs. cont.}, x'(t)\in L^2_{[0,\ell]},\,  x(\ell)=0\right\}.
    \end{aligned}
\end{equation}
Following the steps of Example 7.6 of \cite{ABT} we have
\begin{equation}\label{e7-83}
\begin{aligned}
\dA^\ast x&=i\frac{dx}{dt},\;
\dom(\dA^\ast)&=\left\{x(t)\,\Big|\,x(t)-\text{ abs. cont.}, x'(t)\in L^2_{[0,\ell]}\right\}.\\
\end{aligned}
\end{equation}
Then $\calH_+=\dom(\dA^\ast)=W^1_2$ is the Sobolev space with scalar
product
\begin{equation}\label{e-product}
(x,y)_+=\int_0^\ell x(t)\overline{y(t)}\,dt+\int_0^\ell
x'(t)\overline{y'(t)}\,dt.
\end{equation}
 Construct rigged Hilbert space
$W^1_2\subset L^2_{[0,\ell]}\subset (W_2^1)_-$
and consider operators
\begin{equation}\label{e7-84}
\bA x=i\frac{dx}{dt}+i x(0)\left[\delta(t)-\delta(t-\ell)\right],\quad
\bA^\ast x=i\frac{dx}{dt}+i x(l)\left[\delta(t)-\delta(t-\ell)\right],
\end{equation}
where $x(t)\in W_2^1$, $\delta(t)$, $\delta(t-\ell)$ are delta-functions and elements of 
$(W^1_2)_-$ that generate functionals by the formulas
$(x,\delta(t))=x(0)$ and $(x,\delta(t-\ell))=x(\ell)$. It is easy to see
that
$\bA\supset T\supset \dA$, $\bA^\ast\supset T^\ast\supset \dA,$
and that
$$
\RE\bA x=i\frac{dx}{dt}+\frac{i }{2}(x(0)+x(\ell))\left[\delta(t)-\delta(t-\ell)\right].
$$
Clearly, $\RE\bA$ has its quasi-kernel equal to $A$ in \eqref{e-58'}. Moreover,
$$
\IM\bA x=\left(\cdot,\frac{1}{\sqrt 2}[\delta(t)-\delta(t-\ell)]\right) \frac{1}{\sqrt 2}[\delta(t)-\delta(t-\ell)]=(\cdot,\chi)\chi,
$$
where $\chi=\frac{1}{\sqrt 2}[\delta(t)-\delta(t-\ell)]$.
Now we can build
\begin{equation}\label{e6-125}
\Theta=
\begin{pmatrix}
\bA &K &1\\
&&\\
W_2^1\subset L^2_{[0,\ell]}\subset (W^1_2)_- &{ } &\dC
\end{pmatrix},
\end{equation}
that is a minimal L-system with
\begin{equation}\label{e7-62}
\begin{aligned}
Kc&=c\cdot \chi=c\cdot \frac{1}{\sqrt 2}[\delta(t)-\delta(t-l)], \quad (c\in \dC),\\
K^\ast x&=(x,\chi)=\left(x,  \frac{1}{\sqrt
2}[\delta(t)-\delta(t-l)]\right)=\frac{1}{\sqrt
2}[x(0)-x(l)],\\
\end{aligned}
\end{equation}
and $x(t)\in W^1_2$. In order to find the transfer function of $\Theta$ we begin by evaluating the resolvent of operator $T$ in \eqref{e1-3}. Solving the linear differential equation
of the first order with the initial condition from \eqref{e1-3} yields
\begin{equation}\label{e-resolv-72}
   R_z(T)f= (T-zI)^{-1}f=-i e^{-izt}\int_0^t f(s) e^{izs}\,ds,\quad f\in L^2_{[0,\ell]}.
\end{equation}
Similarly, one finds that
\begin{equation}\label{e-resolv-73}
   R_z(T^*)f= (T^*-zI)^{-1}f=i e^{-izt} \int_t^\ell f(s) e^{izs}\,ds,\quad f\in L^2_{[0,\ell]}.
\end{equation}
We need to extend $R_z(T)$ to $(W^1_2)_-$ to apply it to the vector $g$. We can accomplish this via finding the values of $\hat R_z(T) \delta(t)$ and $\hat R_z(T) \delta(t-l)$ (here $\hat R_z(T)$
is the extended resolvent). We have
$$
\begin{aligned}
(\hat R_z(T)\delta(t),f)&=(\delta(t),R_{\bar z}(T^*)f)=\ol{R_{\bar z}(T^*)f\Big|_{t=0}}
={-i}\int_0^\ell e^{-izs}\ol{f(s)}ds\\
&=(-ie^{-izt},f),\quad f\in L^2_{[0,\ell]},
\end{aligned}
$$
and hence $\hat R_z(T)\delta(t)=-ie^{-izt}$. Similarly, we determine that $\hat R_z(T)\delta(t-l)=0$. Consequently,
$$
\hat R_z(T) g=-\frac{i}{\sqrt 2}e^{-izt}.
$$
Therefore,
\begin{equation}\label{e-87W}
    \begin{aligned}
W_\Theta(z)&=1-2i((T-zI)^{-1}\chi,\chi)=1-2i\left(-\frac{i}{\sqrt 2}e^{-izt},\frac{1}{\sqrt 2}[\delta(t)-\delta(t-\ell)]\right)\\
&=1-(e^{-izt},\delta(t)-\delta(t-\ell))=1-1+e^{-i\ell z}=e^{-i\ell z}.
    \end{aligned}
\end{equation}
This confirms the result of Theorem \ref{t-25} and formula \eqref{e-54} by showing that $W_\Theta(z)=1/S(z)$. The corresponding impedance function is found via \eqref{e6-3-6} and is
$$
V_\Theta(z)=i\frac{e^{-i\ell z}-1}{e^{-i\ell z}+1}.
$$
Direct substitution yields
$$V_{\Theta}(i)=i\frac{e^{\ell}-1}{e^{\ell}+1}=i\frac{1-e^{-\ell}}{1+e^{-\ell}}=i\frac{1-\kappa}{1+\kappa},$$
and thus $V_\Theta(z)\in\sM_\kappa$ with $\kappa=e^{-\ell}$.

\subsection*{Example 2}\label{ex-2}
In this Example  we will rely on the main elements of the construction presented in Example 1 but with some changes. Let $\dA$ and $A$ be still defined by formulas \eqref{e-56} and \eqref{e-58'}, respectively and let $s(z)$ be  the Liv\v{s}ic characteristic function $s(z)$ for the pair $(\dA,A)$ given by \eqref{e1-1}.
We introduce the operator
\begin{equation}\label{e1-3'}
T_0 x=i\frac{dx}{dt},\;
\end{equation}
$$
\dom(T_0)=\left\{x(t)\,\Big|\,x(t)-\text{abs. cont.}, x'(t)\in L^2_{[0,\ell]},\,  x(\ell)=e^{\ell}x(0)\right\}.
$$
It turns out that  $T_0$ is a dissipative extension of $\dA$ parameterized by a von Neumann parameter $\kappa=0$. Indeed, using \eqref{e-57} with \eqref{parpar} again we obtain
\begin{equation}\label{e-70}
x(t)=Ce^{t}-\kappa\, Ce^\ell e^{-t}\in\dom(T),\quad x(\ell)=e^{\ell}x(0),
\end{equation}
yielding $\kappa=0$. Clearly, the triple of operators $( \dot A,T_0, A)$ satisfy the conditions of Hypothesis \ref{setup} but this time, since $\kappa=0$, we have that $S(z)=-s(z)$.

Following the steps of Example 1 we are going to use the triple $(\dA, T_0, A)$ in the construction of an L-system $\Theta_0$. By the direct check one gets
\begin{equation}\label{e1-3star'}
T_0^* x=i\frac{dx}{dt},
\end{equation}
$$
\dom(T)=\left\{x(t)\,\Big|\,x(t)-\text{abs. cont.}, x'(t)\in L^2_{[0,\ell]},\,  x(\ell)=e^{-\ell}x(0)\right\}.
$$
Once again, we have $\dA^*$ defined by \eqref{e7-83}
and  $\calH_+=\dom(\dA^\ast)=W^1_2$ is a  space with scalar
product \eqref{e-product}.
 Consider the operators
\begin{equation}\label{e-73}
\begin{aligned}
\bA_0 x&=i\frac{dx}{dt}+i \frac{x(\ell)-e^\ell x(0)}{e^\ell-1} \left[\delta(t-\ell)-\delta(t)\right],\\
\bA_0^\ast &x=i\frac{dx}{dt}+i \frac{x(0)-e^{\ell} x(\ell)}{e^\ell-1} \left[\delta(t-\ell)-\delta(t)\right],
\end{aligned}
\end{equation}
where $x(t)\in W_2^1$. It is easy to see
that
$\bA\supset T_0\supset \dA$, $\bA^\ast\supset T_0^\ast\supset \dA,$
and
$$
\RE\bA_0 x=i\frac{dx}{dt}-\frac{i }{2}(x(0)+x(\ell))\left[\delta(t-\ell)-\delta(t)\right].
$$
Thus $\RE\bA_0$ has its quasi-kernel equal to $A$ in \eqref{e-58'}. Similarly,
$$
\IM\bA_0 x=\left(\frac{1}{2}\right) \frac{e^\ell+1}{e^\ell-1} (x(\ell)-x(0))\left[\delta(t-\ell)-\delta(t)\right].
$$
Therefore,
$$
\begin{aligned}
\IM\bA_0&=\left(\cdot,\sqrt{\frac{e^\ell+1}{2(e^\ell-1)}}\,\left[\delta(t-\ell)-\delta(t)\right]\right)\sqrt{\frac{e^\ell+1}{2(e^\ell-1)}}\,\left[\delta(t-\ell)-\delta(t)\right]\\
&=(\cdot,\chi_0)\chi_0,
\end{aligned}
$$
where $\chi_0=\sqrt{\frac{e^\ell+1}{2(e^\ell-1)}}\,[\delta(t-\ell)-\delta(t)]$.
Now we can build
\begin{equation*}
\Theta_0= 
\begin{pmatrix}
\bA_0&K_0 &1\\
&&\\
W_2^1\subset L^2_{[0,l]}\subset (W^1_2)_- &{ } &\dC
\end{pmatrix},
\end{equation*}
which is a minimal  L-system with
$K_0 c=c\cdot \chi_0$, $(c\in \dC)$, $K_0^\ast x=(x,\chi_0)$ and $x(t)\in W^1_2$. Following Example 1 we derive
\begin{equation}\label{e-resolv-78}
    \begin{aligned}
   R_z(T_0)&=(T_0-zI)^{-1}f\\
   &=-i e^{-izt}\left(\int_0^t f(s) e^{izs}\,ds+\frac{e^{-i\ell z}}{e^\ell-e^{-i\ell z}}\int_0^l f(s) e^{izs}\,ds\right),
    \end{aligned}
\end{equation}
and
\begin{equation}\label{e-resolv-79}
    \begin{aligned}
  R_z(T_0^*)&= (T_0^*-zI)^{-1}f\\
  &=-i e^{-izt}\left(\int_0^t f(s) e^{izs}\,ds+\frac{e^{-i\ell z}}{e^{-\ell}-e^{-i\ell z}}\int_0^l f(s) e^{izs}\,ds\right),
  \end{aligned}
\end{equation}
for $f\in L^2_{[0,\ell]}$. Then again
$$
\begin{aligned}
(\hat R_z(T_0)\delta(t),f)&=(\delta(t),R_{\bar z}(T_0^*)f)=\ol{R_{\bar z}(T_0^*)f\Big|_{t=0}}
=\frac{i e^{i\ell z}}{e^{-\ell}-e^{i\ell z}}\int_0^\ell e^{-izs}\ol{f(s)}ds\\
&=\frac{i e^{\ell}}{e^{-i\ell z}-e^{\ell}}(e^{-izt},f),\quad f\in L^2_{[0,\ell]}.
\end{aligned}
$$
Similarly,
$$
\begin{aligned}
(\hat R_z(T_0)&\delta(t-\ell),f)=(\delta(t-\ell),R_{\bar z}(T_0^*)f)=\ol{R_{\bar z}(T_0^*)f\Big|_{t=\ell}}
\\&=\frac{i e^{i\ell z}e^{-\ell}}{e^{-\ell}-e^{i\ell z}}\int_0^\ell e^{-izs}\ol{f(s)}ds
=\frac{i}{e^{-i\ell z}-e^{\ell}}(e^{-izt},f),\quad f\in L^2_{[0,\ell]}.
\end{aligned}
$$
Hence,
\begin{equation}\label{e-80}
   \hat R_z(T_0)\delta(t)= \frac{i e^{\ell}}{e^{-i\ell z}-e^{\ell}}\,e^{-izt},\quad
   \hat R_z(T_0)\delta(t-\ell)=\frac{i}{e^{-i\ell z}-e^{\ell}}\,e^{-izt},
\end{equation}
and
$$
\hat R_z(T_0) \chi_0=\hat R_z(T_0)\sqrt{\frac{e^\ell+1}{2(e^\ell-1)}}\left[\delta(t-\ell)-\delta(t)\right]=\sqrt{\frac{e^\ell+1}{2(e^\ell-1)}}\frac{i-i e^{\ell}}{e^{-i\ell z}-e^{\ell}}\,e^{-izt}.
$$
Using techniques of Example 1 one finds the
transfer function of $\Theta_0$ to be
$$
\begin{aligned}
W_{\Theta_0}(z)&=1-2i(\hat R_z(T_0)\chi_0,\chi_0)\\
&=1-2i\left(\sqrt{\frac{e^\ell+1}{2(e^\ell-1)}}\frac{i-i e^{\ell}}{e^{-i\ell z}-e^{\ell}}\,e^{-izt},\sqrt{\frac{e^\ell+1}{2(e^\ell-1)}}[\delta(t-\ell)-\delta(t)]\right)\\
&=1+\frac{e^\ell+1}{e^\ell-1}\left(\frac{e^{\ell}-1}{e^{-i\ell z}-e^{\ell}}\,e^{-izt},\delta(t-\ell)-\delta(t)\right)\\
&=1-\frac{e^\ell+1}{e^\ell-1}\left(\frac{e^{\ell}-1}{e^{-i\ell z}-e^{\ell}}-\frac{(e^{\ell}-1)e^{-iz\ell}}{e^{-i\ell z}-e^{\ell}}\right)\\
&=1+(e^{\ell}+1)\left(\frac{1-e^{-iz\ell}}{e^{-i\ell z}-e^\ell}\right)
\\&=\frac{e^\ell e^{-i\ell z}-1}{e^\ell-e^{-i\ell z}}.
\end{aligned}
$$
This confirms the result of Corollary \ref{c-26} and formula \eqref{e-54} by showing that $W_{\Theta_0}(z)=-1/s(z)$.
The corresponding impedance function is
$$
V_{\Theta_0}(z)=i\frac{e^\ell+1}{e^{\ell}-1}\cdot\frac{e^{-i\ell z}-1}{e^{-i\ell z}+1}.
$$
A quick inspection confirms that $V_{\Theta_0}(i)=i$ and hence $V_{\Theta_0}(z)\in\sM$.

\subsection*{Remark}

We can use  Examples 1 and 2 to illustrate Lemma \ref{l-20} and Theorem \ref{t-22}.
As one can easily tell that the impedance function $V_{\Theta_0}(z)$ from Example 2 above and the impedance function $V_{\Theta}(z)$ from Example 1 are related via \eqref{e-imp-m} with $\kappa=e^{-\ell}$, that is,
$$
V_{\Theta}(z)=\frac{1-e^{-\ell}}{1+e^{-\ell}}\,V_{\Theta_0}(z).
$$

Let $\Theta$ be the L-system of the form \eqref{e6-125} described in Example 1 with the transfer function $W_\Theta(z)$ given by \eqref{e-87W}. It was shown in \cite[Theorem 8.3.1]{ABT} that if one takes a function $W(z)=-W_\Theta(z)$, then $W(z)$ can be realized as a transfer function of another L-system $\Theta_1$ that shares the same main operator $T$ with $\Theta$ and in this case
$$
V_{\Theta_1}(z)=-1/V_{\Theta}(z)=i\frac{e^{-i\ell z}+1}{e^{-i\ell z}-1}.
$$
Clearly, $V_{\Theta_1}(z)$ and $V_{\Theta_0}(z)$ are not related via \eqref{e-imp-m} even though $\Theta_1$ has the same operator $T$ with the same parameter $\kappa=e^{-\ell}$ as in $\Theta$. The reason for that is the fact that the quasi-kernel of the real part of $\bA_1$ of the L-system $\Theta_1$ does not satisfy the conditions of Hypothesis \ref{setup} as indicated by Theorem \ref{t-22}.
\subsection*{Example 3}\label{ex-3}
In this Example we are going to extend the construction of Example 2 to obtain a family of L-systems $\Theta_0(\beta)$ described in \eqref{e-62-beta}. Let $\dA$  be  defined by formula \eqref{e-56} but the operator $A$ be an arbitrary self-adjoint extension of $\dA$. It is known then \cite{AG93} that all such operators $A$ are described with the help of a unimodular parameter $\mu$ as follows
\begin{equation}\label{e-58-mu}
    \begin{aligned}
A x&=i\frac{dx}{dt},\\
\dom(A)&=\left\{x(t)\,\Big|\,x(t) \in \dom(\dA^*),\, \mu x(\ell)+x(0)=0,\,|\mu|=1\right\}.\\
    \end{aligned}
\end{equation}
In order to establish the connection between the boundary value $\mu$ in \eqref{e-58-mu} and the von Neumann parameter $U$ in \eqref{DOMHAT} we follow the steps similar to Example 1 to guarantee that $g_++ U g_-\in\dom(A)$, where $g_\pm$ are given by \eqref{e-57}. Quick set of calculations yields
\begin{equation}\label{e-99-mu}
    U=-\frac{1+\mu e^\ell}{\mu+e^\ell}.
\end{equation}
For this value of $U$ we set the value of $\beta$ so that $U=e^{2i\beta}$, where $\beta\in[0,\pi)$ and thus establish the link between the parameters $\mu$ and $\beta$ that will be used to construct the family $\Theta_0(\beta)$. In particular, we note that $\beta=0$ if and only if $\mu=-1$.

Once again, having $\dA^*$ defined by \eqref{e7-83} and  $\calH_+=\dom(\dA^\ast)=W^1_2$  a  space with scalar
product \eqref{e-product}, consider the following operators
\begin{equation}\label{e-73-mu}
\begin{aligned}
\bA_0(\beta) x&=i\frac{dx}{dt}+i \frac{\bar\mu}{\bar\mu+e^{-\ell}}(x(0)-e^{-\ell}x(\ell)) \left[\mu\delta(t-\ell)+\delta(t)\right],\\
\bA_0^\ast(\beta) &x=i\frac{dx}{dt}+i \frac{1}{\mu+e^{-\ell}}(e^{-\ell}x(0)-x(\ell)) \left[\mu\delta(t-\ell)+\delta(t)\right],
\end{aligned}
\end{equation}
where $x(t)\in W_2^1$. It is immediate that $\bA\supset T_0\supset \dA$, $\bA^\ast\supset T_0^\ast\supset \dA,$  where $T_0$ and $T_0^*$ are given by \eqref{e1-3'} and \eqref{e1-3star'}. Also, as one can easily see, when $\beta=0$ and consequently $\mu=-1$, the operators $\bA_0(0)$ and $\bA_0^\ast(0)$ in \eqref{e-73-mu} match the corresponding pair $\bA_0$ and $\bA_0^*$ in \eqref{e-73}. By performing direct calculations we obtain
$$
\RE\bA_0(\beta) x=i\frac{dx}{dt}+\frac{i }{2}(\nu x(\ell)+x(0))\left[\mu\delta(t-\ell)+\delta(t)\right],
$$
where
\begin{equation}\label{e-101-nu}
\nu=\frac{2\mu e^{-\ell}+e^{-2\ell}+1}{\mu+2e^{-\ell}+\mu e^{-2\ell}},
\end{equation}
and $|\nu|=1.$ Consequently, $\RE\bA_0$ has its quasi-kernel
\begin{equation}\label{e-102-qk}
    \hat A_0(\beta)=i\frac{dx}{dt},\quad
\dom(A)=\left\{x(t)\,\Big|\,x(t) \in \dom(\dA^*),\, \nu x(\ell)+x(0)=0\right\}.
\end{equation}
Moreover,
$$
\IM\bA_0(\beta) x=\left(\frac{1}{2}\right) \left(\frac{1-e^{-2\ell}}{|\mu+e^{-2\ell}|}\right) (\bar\mu x(\ell)+x(0))\left[\mu\delta(t-\ell)+\delta(t)\right].
$$
Therefore,
$$
\begin{aligned}
\IM\bA_0(\beta)&=\left(\cdot,\frac{\sqrt{1-e^{-2\ell}}}{\sqrt2 |\mu+e^{-2\ell}| }\,\left[\mu\delta(t-\ell)+\delta(t)\right]\right)\frac{\sqrt{1-e^{-2\ell}}}{\sqrt2 |\mu+e^{-2\ell}| }\,\left[\mu\delta(t-\ell)+\delta(t)\right]\\
&=(\cdot,\chi_0(\beta))\chi_0(\beta),
\end{aligned}
$$
where $\chi_0(\beta)=\sqrt{\frac{e^\ell+1}{2(e^\ell-1)}}\,[\delta(t-\ell)-\delta(t)]$.
Now we can compose our one-parametric L-system family
\begin{equation*}
\Theta_0(\beta)= 
\begin{pmatrix}
\bA_0(\beta)&K_0(\beta) &1\\
&&\\
W_2^1\subset L^2_{[0,l]}\subset (W^1_2)_- &{ } &\dC
\end{pmatrix},
\end{equation*}
where $K_0(\beta) c=c\cdot \chi_0(\beta)$, $(c\in \dC)$, $K_0^\ast(\beta) x=(x,\chi_0(\beta))$ and $x(t)\in W^1_2$. Using techniques of Example 2 one finds the
transfer function of $\Theta_0(\beta)$ to be
$$
\begin{aligned}
W_{\Theta_0(\beta)}(z)&=1-2i(\hat R_z(T_0)\chi_0(\beta),\chi_0(\beta))=\left(\frac{e^\ell+\mu}{\mu e^\ell+1} \right)\frac{e^\ell e^{-i\ell z}-1}{e^\ell-e^{-i\ell z}}.
\end{aligned}
$$
The corresponding impedance function is again found via \eqref{e6-3-6}
$$
V_{\Theta_0(\beta)}(z)=i\frac{(\bar\mu e^{-i\ell z}-1)(e^{2\ell}+1)+2e^\ell e^{-i\ell z}-2\bar\mu e^\ell }{(\bar\mu e^{-i\ell z}+1)(e^{2\ell}-1)}.
$$
A quick inspection confirms that $V_{\Theta_0(\beta)}(i)=i$ and hence $V_{\Theta_0(\beta)}(z)$ belongs to the Donoghue class $\sM$ for all $\beta\in[0,\pi)$ (equivalently $|\mu|=1$). Also, one can see that if $\beta=0$ and consequently $\mu=-1$ the conditions of Hypothesis \ref{setup} are satisfied and the L-system $\Theta_0(0)$ coincides with the L-system $\Theta_0$ of Example 2 and so do its  transfer and impedance functions.

\subsection*{Example 4}\label{ex-4}
In this Example  we will generalize the results obtained in Examples 1 and 2. Once again, let $\dA$ and $A$ be  defined by formulas \eqref{e-56} and \eqref{e-58'}, respectively and let $s(z)$ be  the Liv\v{s}ic characteristic function $s(z)$ for the pair $(\dA,A)$ given by \eqref{e1-1}.
We introduce a one-parametric family of operators
\begin{equation}\label{e3-33}
    \begin{aligned}
T_\rho x&=i\frac{dx}{dt},\;
\dom(T_\rho)=\left\{x(t)\,\Big|\,x(t)-\text{abs. cont.}, x'(t)\in L^2_{[0,\ell]},\,  x(\ell)=\rho x(0)\right\}.
    \end{aligned}
\end{equation}
We are going to select the values of boundary parameter $\rho$ in a way that will make $T_\rho$ compliant with Hypothesis \ref{setup}. By performing the direct check we conclude that $\IM (T_\rho f,f)\ge0$ for $f\in \dom(T_\rho)$ if $|\rho|>1$. This will guarantee that $T_\rho$ is a  dissipative extension of $\dA$ parameterized by a von Neumann parameter $\kappa$. For further convenience we assume that $\rho\in\dR$. To find the connection between $\kappa$ and $\rho$ we  use \eqref{e-57} with \eqref{parpar} again to obtain
\begin{equation}\label{e-70'}
x(t)=Ce^{t}-\kappa\, Ce^\ell e^{-t}\in\dom(T),\quad x(\ell)=\rho x(0).
\end{equation}
Solving \eqref{e-70'} in two ways  yields
\begin{equation}\label{e-kappa-rho}
\kappa=\frac{\rho-e^\ell}{\rho e^\ell-1}\quad\textrm{ and }\quad \rho=\frac{\kappa-e^\ell}{\kappa e^\ell-1}.
\end{equation}
Using the first of relations \eqref{e-kappa-rho} to find which values of $\rho$ provide us with $0\le\kappa<1$ we obtain
\begin{equation}\label{e-96-rho}
    \rho\in (-\infty,-1)\cup [e^\ell,+\infty).
\end{equation}
 Now assuming \eqref{e-96-rho} we can acknowledge that  the triplet of operators $( \dot A,T_\rho, A)$ satisfy the conditions of Hypothesis \ref{setup}. Following Examples 1 and 2, we are going to use the triplet $(\dA, T_\rho, A)$ in the construction of an L-system $\Theta_\rho$. By the direct check we have
\begin{equation}\label{e1-3star''}
T_\rho^* x=i\frac{dx}{dt},
\end{equation}
$$
\dom(T_\rho)=\left\{x(t)\,\Big|\,x(t)-\text{abs. cont.}, x'(t)\in L^2_{[0,\ell]},\,  \rho x(\ell)= x(0)\right\}.
$$
Once again, we have $\dA^*$ defined by \eqref{e7-83} and  $\calH_+=\dom(\dA^\ast)=W^1_2$ is a  space with scalar
product \eqref{e-product}.
 Consider the operators
\begin{equation}\label{e-73'}
\begin{aligned}
\bA_\rho x&=i\frac{dx}{dt}+i \frac{x(\ell)-\rho x(0)}{\rho-1} \left[\delta(t-\ell)-\delta(t)\right],\\
\bA_\rho^\ast &x=i\frac{dx}{dt}+i \frac{x(0)-\rho x(\ell)}{\rho-1} \left[\delta(t-\ell)-\delta(t)\right],
\end{aligned}
\end{equation}
where $x(t)\in W_2^1$. One easily checks that since $\IM\rho=0$, then $\bA_\rho^*$ is the adjoint to $\bA_\rho$ operator.  Evidently,
that
$\bA\supset T_\rho \supset \dA $, $\bA^\ast\supset T_\rho ^\ast\supset \dA ,$
and
$$
\RE\bA_\rho  x=i\frac{dx}{dt}-\frac{i }{2}(x(0)+x(\ell))\left[\delta(t-\ell)-\delta(t)\right].
$$
Thus $\RE\bA_\rho $ has its quasi-kernel equal to $A$ defined in \eqref{e-58'}. Similarly,
$$
\IM\bA_\rho  x=\left(\frac{1}{2}\right) \frac{\rho+1}{\rho-1} (x(\ell)-x(0))\left[\delta(t-\ell)-\delta(t)\right].
$$
Therefore,
$$
\begin{aligned}
\IM\bA_\rho &=\left(\cdot,\sqrt{\frac{\rho+1}{2(\rho-1)}}\,\left[\delta(t-\ell)-\delta(t)\right]\right)\sqrt{\frac{\rho+1}{2(\rho-1)}}\,\left[\delta(t-\ell)-\delta(t)\right]\\
&=(\cdot,\chi_\rho )\chi_\rho ,
\end{aligned}
$$
where $\chi_\rho =\sqrt{\frac{\rho+1}{2(\rho-1)}}\,[\delta(t-\ell)-\delta(t)]$.
Now we can build
\begin{equation*}
\Theta_\rho =
\begin{pmatrix}
\bA_\rho &K_\rho  &1\\
&&\\
W_2^1\subset L^2_{[0,l]}\subset (W^1_2)_- &{ } &\dC
\end{pmatrix},
\end{equation*}
which is a minimal  L-system with
$K_\rho  c=c\cdot \chi_\rho $, $(c\in \dC)$, $K_\rho ^\ast x=(x,\chi_\rho )$ and $x(t)\in W^1_2$. Evaluating the transfer function $W_{\Theta_\rho}(z)$ resembles the steps performed in Example 2.
We have
\begin{equation}\label{e-resolv-rho}
    \begin{aligned}
   R_z(T_\rho)&=(T_\rho-zI)^{-1}f\\
   &=-i e^{-izt}\left(\int_0^t f(s) e^{izs}\,ds+\frac{e^{-i\ell z}}{\rho-e^{-i\ell z}}\int_0^l f(s) e^{izs}\,ds\right).
    \end{aligned}
\end{equation}
This leads to
$$
\hat R_z(T_\rho) \chi_\rho=\hat R_z(T_\rho)\sqrt{\frac{\rho+1}{2(\rho-1)}}\left[\delta(t-\ell)-\delta(t)\right]=i \sqrt{\frac{\rho+1}{2(\rho-1)}}\left(\frac{1-\rho}{e^{-i\ell z}-\rho}\right)\,e^{-izt},
$$
and eventually to
$$
W_{\Theta_\rho }(z)=1-2i(\hat R_z(T_\rho )\chi_\rho ,\chi_\rho )=\frac{\rho e^{-i\ell z}-1}{\rho-e^{-i\ell z}}.
$$
Evaluating the impedance function $V_{\Theta_\rho }(z)$ results in
$$
V_{\Theta_\rho}(z)=i\frac{\rho+1}{\rho-1}\cdot\frac{1-e^{-i\ell z}}{1+e^{-i\ell z}}.
$$
Using direct calculations and \eqref{e-kappa-rho} gives us
$$
\frac{\rho+1}{\rho-1}=\frac{1-\kappa}{1+\kappa}\cdot \frac{e^\ell+1}{e^\ell-1},
$$
and thus
$$
V_{\Theta_\rho}(z)=\frac{1-\kappa}{1+\kappa}\,V_{\Theta_0}(z),
$$
which confirms the result of Lemma \ref{l-20}.

\appendix
\section{Rigged Hilbert spaces}\label{A1}

In this Appendix we are going to explain the construction and basic geometry of rigged Hilbert spaces.

We start with a Hilbert space $\mathcal H$ with inner product
$(x,y)$ and norm $\|\cdot\|$. Let  $\mathcal H_+$ be a dense in
$\calH$ linear set that is a Hilbert space itself with respect to
another inner product $(x,y)_+$ generating the norm $\|\cdot\|_+$.
We assume that $\|x\|\le\|x\|_+$, ($x\in\calH_+$), i.e., the norm
$\|\cdot\|_+$ generates a stronger than  $\|\cdot\|$ topology in
$\calH_+$. The space $\mathcal H_+$ is called the \textit{space with the
positive norm}.

Now let $\mathcal H_-$ be a space dual to $\mathcal H_+$. It means
that $\mathcal H_-$ is a space of linear functionals defined on
$\mathcal H_+$ and continuous with respect to $\|\cdot\|_+$. By the
$\|\cdot\|_-$ we denote the norm in $\mathcal H_-$ that has a form
\[
\|h\|_{-}=\sup\limits_{u\in\cH_+}\frac{|(h,u)|}{\|u\|_+},\; h\in\cH.
\]
 The value of a
functional $f\in\mathcal H_-$ on a vector $u\in\mathcal H_+$ is
denoted by $(u,f)$. The space $\mathcal H_-$ is called the
\textit{space with the negative norm}.

Consider an embedding operator $\sigma:\calH_+\mapsto\calH$ that embeds $\calH_+$ into
$\calH$. Since $\|\sigma f\|\le\|f\|_+$ for all $f\in\calH_+$, then
$\sigma\in[\calH_+,\calH]$. The adjoint operator $\sigma^*$ maps
$\calH$ into $\calH_-$ and satisfies the condition $\|\sigma^*
f\|_-\le\|f\|$ for all $f\in\calH$. Since $\sigma$ is a monomorphism
with a $(\cdot)$-dense range, then $\sigma^*$ is a  monomorphism
with $(-)$-dense range. By identifying $\sigma^* f$ with $f$
($f\in\calH$) we can consider $\calH$  embedded in $\calH_-$ as a
$(-)$-dense set and $\|f\|_-\le\|f\|$. Also, the relation
$$(\sigma f,h)=(f,\sigma^* h),\qquad f\in\calH_+,\, h\in\calH,$$
implies that the value of the functional $\sigma^* h\in\calH$
calculated at a vector $f\in\calH_+$ as $(f,\sigma^*h)$ corresponds
to the value $(f,h)$ in the space $\calH$.

It follows from the Riesz representation theorem that  there exists
an isometric  operator $\calR$ which maps $\mathcal H_-$ onto
$\mathcal H_+$ such
 that   $(f,g)=(f,\calR g)_+$ ($\forall f\in\calH_+$, $g\in\calH_-$) and
 $\|\calR g\|_+=\| g\|_-$. Now we can turn $\calH_-$ into a Hilbert
 space by introducing $(f,g)_-=(\calR f,\calR g)_+$. Thus,
\begin{equation}\label{e3-4}
\aligned (f,g)_-=(f,\calR g)=(\calR f,g)=(\calR f,\calR g)_+,\qquad
(f,g\in \mathcal H_-),\\
(u,v)_+=(u,\calR^{-1} v)=(\calR^{-1} u,v)=(\calR^{-1} u,\calR^{-1}
v)_-,\qquad (u,v\in \mathcal H_+).
\endaligned
\end{equation}
The operator $\calR$ (or $\calR^{-1}$) will be called the
\textit{Riesz-Berezansky operator}.  We note that $\calH_+$ is also dual to $\calH_-$. Applying the above reasoning, we
define a triplet $\mathcal H_+\subset \mathcal H\subset \mathcal
H_-$ to be called the \textit{rigged Hilbert space} \cite{Ber63}, \cite{Ber}.

Now we explain how to construct a rigged Hilbert space using a symmetric operator. Let  $\dA$ be a closed symmetric operator whose domain
$\dom(\dA)$ is not assumed to be dense in $\calH$. Setting
$\ol{\dom(\dA)}=\calH_0$, we can consider $\dA$ as a densely defined
operator from $\calH_0$ into $\calH$. Clearly, $\dom(\dA^*)$ is
dense in $\calH$ and $\ran(\dA^*)\subset\calH_0$.
 We introduce a new Hilbert space $\mathcal H_+=\dom(\dot
A^*)$  with inner product
\begin{equation}\label{e3-5}
(f,g)_+=(f,g)+(\dot A^*f,\dot A^*g),\qquad (f,g\in \mathcal H_+),
\end{equation}
and then construct the operator generated  rigged Hilbert space $\mathcal H_+\subset
\mathcal H\subset \mathcal H_-$.



\begin{thebibliography}{1}



\bibitem{AG93}
N.~I.~Akhiezer, I.~M.~Glazman, \newblock {\em Theory of linear operators.} \newblock Pitman {A}dvanced {P}ublishing {P}rogram, 1981.


\bibitem{AWT}
A.~ Aleman, R.~T.~W.~Martin, W.~T. ~Ross,
 {\it On a theorem of Liv\v{s}ic},
  J. Funct. Anal., {\bf 264}, 999--1048, (2013).

\bibitem{ABT}
{Yu.~Arlinski\u{\i}, S.~Belyi, E.~Tsekanovski\u{\i}},
\textit{Conservative Realizations  of Herglotz-Nevanlinna functions.}
{Operator Theory: Advances and Applications}, {Vol. 217}, {Birkh\"auser},  {2011}.

\bibitem{ArTs03}
Yu. Arlinski\u{\i},  E.~Tsekanovski\u{\i}, {\it Constant $J$-unitary factor and opera\-tor-valued transfer functions}. In: Dynamical systems and
differential equations, Discrete Contin. Dyn. Syst., Wilmington, NC, 48--56, (2003).


\bibitem{BT3}
{S.~Belyi, E.~Tsekanovski\u{\i}}, \textit{Realization theorems for operator-valued $R$-functions}, {Oper. Theory Adv. Appl.}, \textbf{98}, {55--91}, {(1997)}.


\bibitem{Ber63}
{Yu.~M.~Berezansky},  \textit{Spaces with negative norm},  {Uspehi Mat. Nauk}, {vol. 18}, no. 1 (109) 63--96, (1963) (Russian).

\bibitem{Ber}
{Yu.~M.~Berezansky},  \textit{Expansion in eigenfunctions of self-adjoint operators}, {vol. 17}, {Transl. Math. Monographs}, {AMS}, {Providence}, {1968}.


\bibitem{Br}
M.~S.~Brodskii, {\it Triangular and Jordan representations of linear operators.} Translations of Mathematical Monographs,
 Vol. 32. American Mathematical Society, Providence, R.I., 1971.


\bibitem{BL58}
M.~S.~Brodskii, M.~S.~Liv\v{s}ic, \textit{Spectral analysis of non-self-adjoint operators and intermediate systems}, Uspehi Mat. Nauk (N.S.) {\bf 13},  no. 1 (79), 3--85, (1958) (Russian).

\bibitem{DM} V.~A.~Derkach,  M.~M.~Malamud,
\textit{Generalized resolvents and the boundary value problems for Hermitian operators with gaps},  J. Funct. Anal. {\bf 95}, 1--95, (1991).

\bibitem {D}
 W.~F.~Donoghue, \textit{On perturbation of spectra}, Commun. Pure and Appl. Math. {\bf 18},  559--579, (1965).

 \bibitem{GMT}
F.~Gesztesy, K.~A.~Makarov, E.~Tsekanovskii, \textit{An addendum to Krein's formula}, J. Math. Anal. Appl. {\bf 222}, 594--606, (1998).

\bibitem{GT}
F.~Gesztesy, E.~Tsekanovskii, \textit{On Matrix-Valued Herglotz Functions}, Math. Nachr. {\bf 218}, 61--138, (2000).


\bibitem{Koch} A.~N.~Kochubei, {\it Characteristic functions of symmetric operators and their extensions}, Izv. Akad. Nauk Armyan. SSR Ser. Mat. 15, no. 3, 219--232, (1980) (Russian).


\bibitem{L}
M.~S.~Liv\v{s}ic, \textit{On a class of linear operators in Hilbert space},
Mat. Sbornik  (2), {\bf 19}, 239--262 (1946) (Russian); English transl.:  Amer. Math. Soc. Transl., (2), {\bf 13}, 61--83, (1960).



\bibitem{L1}
M.~S.~Liv\v{s}ic, \textit{On spectral decomposition of linear non-self-adjoint operators},
 Mat. Sbornik (76) {\bf 34},  145--198, (1954) (Russian);  English transl.:  Amer. Math. Soc. Transl. (2) {\bf 5}, 67--114, (1957).


\bibitem{Lv2}
M.~S.~Liv\v{s}ic, \textit{Operators, oscillations, waves.} Moscow, Nauka, 1966 (Russian); English transl.: Amer. Math. Soc. Transl., Vol. 34., Providence, R.I., (1973).



\bibitem{MT-S}
{K.~A.~ Makarov, E.~Tsekanovski\u i}, \textit{On the Weyl-Titchmarsh and Liv\v{s}ic  functions}, Proceedings of Symposia in Pure Mathematics, Vol. 87, 291--313, American Mathematical Society, (2013).


\bibitem{MT10}
{K.~A.~ Makarov, E.~Tsekanovski\u i}, \textit{On the addition and  multiplication theorems}, Operator Theory: Advances and Applications, {\bf 244}, 315--339, (2015).



\bibitem{Na68}{M.~A.~Naimark}, \textit{Linear Differential Operators II.}, {F. Ungar Publ.}, {New York}, {1968}.


\bibitem{St68}
A.~V.~Shtraus, \textit{On the extensions and the characteristic function of a symmetric operator}, Izv. Akad. Nauk SSR, Ser. Mat., {\bf 32}, 186--207, (1968).

\bibitem{T69}
E.~Tsekanovski\u i, \textit{The description and the uniqueness of generalized extensions of quasi-Hermitian operators}. (Russian) Funkcional. Anal. i Prilozen., \textbf{3}, No.1, 95--96, (1969).



\bibitem{TSh1}
{E.~Tsekanovski\u i, ~Yu. ~\u Smuljan,}   \textit{The theory of bi-extensions of operators on rigged Hilbert spaces. Unbounded operator colligations and characteristic functions},
 {Russ. Math. Surv.}, {\bf 32}, {73--131},  {(1977)}.

\end{thebibliography}
\end{document}